\documentclass{amsart}
\usepackage[T1]{fontenc}
\usepackage{amsfonts,amssymb,amsmath,mathrsfs}
\usepackage{verbatim}
\usepackage{ifthen}
\newcommand{\e}{\mathrm{e}}
\newcommand{\im}{\mathrm{i}}
\newcommand{\dif}{\mathrm{d}}

\newcommand{\Z}{\mathbb{Z}}
\newcommand{\D}{\mathbb{D}}
\newcommand{\abs}[1]{|#1|}
\newcommand{\Abs}[1]{\left|#1\right|}
\newcommand{\norm}[1]{\|#1\|}
\newcommand{\Norm}[1]{\left\|#1\right\|}
\newcommand{\inner}[2]{\left\langle #1|#2 \right\rangle}
\newcommand{\seq}[1]{(#1)}

\newcommand{\set}[1]{\big\{#1\big\}}
\newcommand{\Set}[1]{\left\{#1\right\}}

\newcommand{\N}{\mathbb{N}}
\newcommand{\R}{\mathbb{R}}
\newcommand{\T}{\mathbb{T}}
\newcommand{\C}{\mathbb{C}}

\newcommand{\Bigoh}[1]{\mathcal{O} \left( #1 \right)}

\newcommand{\Littleoh}[1]{o \left( #1 \right)}
\newcommand{\Hp}{\mathscr{H}}

\newcommand{\supp}{\mathrm{supp \;}}

\allowdisplaybreaks

\renewcommand{\Re}{\operatorname{Re}}

\newtheorem{theorem}{Theorem}
\newtheorem{lemma}{Lemma}

\newtheorem*{thm}{Theorem}

\title[Hardy spaces of Dirichlet series]{On the boundary behaviour of the Hardy spaces of Dirichlet series
and a frame bound estimate}

\author{Jan-Fredrik Olsen}
\address{Department of Mathematical Sciences, Norwegian University of
Science and Technology (NTNU), NO-7491 Trondheim, Norway}
\email{janfreol@math.ntnu.no}

\author{Eero Saksman}
\address{Department of Mathematics and Statistics, University of Helsinki, P. O. Box 68 (Gustaf Hällströmin Katu 2B), FI-00014 University of Helsinki, Finland}
\email{eero.saksman@helsinki.fi}

\thanks{The first author is supported by the Research Council of Norway grant 160192/V30. 
}

\thanks{The second author was supported by the Academy of Finland projects no. 113826 and 118765, and by the Finnish Centre of Excellence in Analysis and Dynamics Research.}

\subjclass[2000]{30B50 (primary), 42B30,  42C15, 46E15 (secondary)}

\begin{document}

\begin{abstract}
	A range of Hardy-like spaces of ordinary Dirichlet series, called the Dirichlet-Hardy spaces $\Hp^p$, $p \geq 1$, have been 	the focus of increasing interest among researchers 	following a paper of Hedenmalm, Lindqvist and Seip \cite{hls97}.
	The Dirichlet series in these spaces converge on a certain half-plane,
	where one may also define the classical Hardy spaces $H^p$.
	In this paper, we compare the boundary behaviour of elements in $\Hp^p$ and $H^p$.
	Moreover, Carleson measures of the spaces $\Hp^p$ are studied.
	Our main result shows that for certain cases the following statement holds true.
	Given an interval on the boundary of the half-plane of definition and a
	function in the classical Hardy space, it possible to find a function in the corresponding
	Dirichlet-Hardy space such that their difference has an analytic continuation across this
	interval.
\end{abstract}

\maketitle

\section{introduction}

%
%
An ordinary Dirichlet series is a function of the form 
$\sum a_n n^{-s}$ with $(a_n)_{n \in \N} \subset \C$, where $s = \sigma + \im t$ denotes the complex variable. 
In the last decade there has been interest in a range of Hardy spaces of ordinary Dirichlet series \cite{hls97,hls99,gordon_hedenmalm99,bayart02,hedenmalm_saksman03,konyagin_queffelec02,bayart03,helson05a,saksman_seip08,seip08}, which for $p \in [1,\infty)$ are defined to be the closure of finite Dirichlet polynomials in the norm 
\begin{equation} \label{limit p norm}
	\lim_{T \rightarrow \infty} \left( \frac{1}{2T} \int_{-T}^T
	\Abs{ \sum_{n\in \N} a_n n^{-\im t} }^p \dif t \right)^{1/p}.
\end{equation}
The spaces are usually called the Dirichlet-Hardy spaces and we denote them by $\Hp^p$. They were defined and studied in \cite{hls97} for $p=2$, and in \cite{bayart02} for other values of $p.$
The function theory of these spaces has proven difficult and merits interest. 
In fact, the most useful definition is in terms of the Hardy spaces on the infinite dimensional torus $H^P(\T^\infty)$. As may be expected, the case $p=2$ is the most tractable since \eqref{limit p norm} coincides with
\begin{equation*}
	\Big\|\sum_{n \in \N} a_n n^{-s}\Big\|_{\Hp^2} =  \Big(\sum_{n \in \N} \abs{a_n}^2\Big)^{1/2}.
\end{equation*}
This is easily verified to be a norm, and moreover,
by the Cauchy-Schwarz inequality it is seen that $\Hp^2$ defines a Hilbert space of functions analytic on the half-plane $\C_{1/2} = \set{\sigma > 1/2}$. 
This domain of analyticity is seen to be the biggest possible by considering the translates $\zeta(s + \sigma)$ of the Riemann zeta function, $\zeta(s) = \sum_{n \in \N} n^{-s}$.
F.~Bayart \cite{bayart02} observed that by a result due to B.~J.~Cole and T.~W.~Gamelin \cite{cole_gamelin86} this holds true for all $p\in [1,\infty)$.

This raises the question of the boundary behaviour of functions in the Dirichlet-Hardy spaces along the abscissa $\sigma = 1/2$.
First considered implicitly by H.~L.~Montgomery, this question was adressed in relation with the space $\Hp^2$ by H.~Hedenmalm, P.~Lindqvist and K.~Seip.
\begin{thm}[\cite{montgomery94},  \cite{hls97}]
	For $F \in \Hp^2$ and every bounded interval $I$ there exists a 
	constant $C>0$, depending only  on the length of $I$, such that 
	\begin{equation} \label{embedding}
		\lim_{\sigma \rightarrow \frac{1}{2}^+ }\int_I \Abs{F\left(\sigma + \im t \right)}^2  \dif t \leq C \norm{F}^2_{\Hp^2}.
	\end{equation}
\end{thm}
We consider this result to be the starting point of our investigation since it opens the door to comparing the spaces $\Hp^p$ to the classical Hardy spaces on the half-plane $\C_{1/2}$,  which are  denoted  by $H^p(\C_{1/2})$. These spaces consist of functions analytic on this half-plane and finite in the norm
\begin{equation*}
	\norm{f}_{H^p(\C_{1/2})} = \sup_{\sigma > 1/2}  \Big(  \int_\R \Abs{f \left(\sigma + \im t \right)}^p \dif t \Big)^{1/p}.
\end{equation*}
The function theory of these spaces is very rich and they are considered to be well understood.
See for instance \cite{duren70,garnett81,koosis98,martinez-avendano_rosenthal07}. 
In particular, functions in the spaces $H^p(\C_{1/2})$ have non-tangential limits almost everywhere on the boundary $\sigma= 1/2$. 
A trick, well-known to researchers in the field, is to apply the embedding \eqref{embedding} to the inequality
\begin{equation} \label{introduction trick}
	 \int_{\R} \Abs{\frac{F\left(\sigma + \im t\right)}{\sigma + \im t}}^2 \dif t
	 \lesssim
	 \sum_{k \in \Z} \frac{1}{k^2 + 1} \int_{k}^{k+1} \Abs{F\left( \sigma + \im t \right)}^2 \dif t.
\end{equation}
It follows that if $F \in \Hp^2$ then $F/s \in H^2(\C_{1/2})$. 
Hence, functions in $\Hp^2$ have non-tangential boundary limits that are locally $L^2$ integrable.
We remark that in \cite{hedenmalm_saksman03} it was shown that for functions in $\Hp^2$ the Dirichlet series converges point-wise almost everywhere on the abscissa $\sigma = 1/2$. 

Our purpose in the present paper is to investigate more deeply the singular behaviour of functions in the spaces $\Hp^p$. Especially, we are able to compare in a quantitative manner the local behaviour near the boundary $\sigma = 1/2$ of elements in $\Hp^2$ to that of functions in the standard Hardy space $H^2$.
Surprisingly, it turns out that functions in these spaces are locally almost alike.
In the next section, we state our results, which are theorems 1 through \ref{sobolev matching real parts}, and give some of the proofs. The proof of our main result, Theorem \ref{main result}, is rather long; after some preliminaries on frame theory in section \ref{section: preliminaries} it is given in sections \ref{main result section}, \ref{proof of lemma example} and \ref{proof of lemma construction}. 
Finally, section \ref{section: mccarthy spaces} is dedicated to the proof of Theorem \ref{sobolev matching real parts}.

We remark that in our proofs we shall also use techniques and arguments from frame theory, the theory of entire functions and operator theory.

\section{Results} \label{results}
We begin with some notation. For a bounded interval $I\subset \R$ we let $\C_I = \set{ s \in \C : \im (s - 1/2) \notin  \R \backslash I }$, i.e. it is the complex plane with two rays on the abscissa $\sigma = 1/2$ removed. We denote by $\mathrm{Hol}(\C_I)$ the set of functions holomorphic in $\C_I$. Note that $f(x) \lesssim g(x)$ 
is taken to mean that there exists a constant $C>0$ such that $f(x) \leq Cg(x)$ for all $x$.
We use the convention $\norm{f}_{L^2(I)}^2 = \int_I \abs{g(t)}^2 \dif t$
and stress that throughout this paper we view $L^2(I)$ as the subspace of $L^2(\R)$ that consists of functions with support in $I$.

\subsection{The case $p=2$}\label{case2}
Our main result concerns the flexibility of the boundary functions of the Dirichlet series in  the space $\Hp^2$.
We refer to the introduction for the definition.
\begin{theorem} \label{main result}
	Let $I \subset \R$ be a bounded interval. Then for every $f \in H^2(\C_{1/2})$ there exist
	$F \in \Hp^2$ 
	and $\phi \in \mathrm{Hol}(\C_I)$ with $\Re \phi(1/2 + \im t) = 0$ on $I$ such that 
	$f = F + \phi$.
	In particular,
	there exists a unique $F \in \Hp^2$ of minimal norm satisfying this. For this function the following holds:
	\begin{enumerate}
		\item There are constants $c_I, C_I >0$, independent of $f$, such that
		\begin{equation*}
			 c_I \norm{\Re f(1/2 + \im t)}_{L^2(I)}^2 \leq \norm{F}_{\Hp^2}^2 \leq C_I \norm{f}_{H^2(\C_{1/2})}^2.
		\end{equation*}
		\item Given a bounded subset $\Gamma \subset \C_I$ at a positive distance
		from $\C \backslash \C_I$ then there exists a constant $D_{\Gamma,I}>0$, independent of $f$, such that 
		\begin{equation*}
			\norm{\phi}_{L^\infty(\Gamma)}^2 \leq D_{\Gamma,I} \norm{f}^2_{H^2(\C_{1/2})}.
		\end{equation*}
	\end{enumerate}
	Moreover, given $\epsilon >0$ and $\abs{I}>1$, the smallest possible $C_I$ satisfies
	\begin{equation} \label{onto: corollary asymptotics}
	 	   \abs{I}^{(1-\epsilon)\frac{\abs{I}}{\pi} \log \pi} \lesssim C_I \lesssim \abs{I}^{(1+\epsilon)\frac{6\abs{I}}{\pi}\log 2},
	\end{equation}
	where the implicit constants depend only on $\epsilon>0$. Also, the largest possible $c_I$ satisfies
	\begin{equation*}
		\frac{1}{\abs{I} + 14} \leq c_I \leq \frac{1}{\abs{I}}.
	\end{equation*}
	Finally,
	\begin{equation*}
		\frac{1}{\pi} \leq \liminf_{\abs{I}\rightarrow 0} C_I \leq \limsup_{\abs{I} \rightarrow 0} C_I \leq \frac{2}{\pi}, 
	\end{equation*}
	and
	\begin{equation*}
	 	\frac{1}{\pi} \leq \liminf_{\abs{I}\rightarrow 0} c_I \leq	\limsup_{\abs{I} \rightarrow \infty} c_I \leq \frac{2}{\pi}.
	\end{equation*}
\end{theorem}
We remark that by keeping track of the constants in the proof, it is possible to show that for $\abs{I} \geq  1$, 
\begin{equation} \label{main theorem quantitative bound}
	 		D_{\Gamma,I}\lesssim C_I  \left\{ 1 + \frac{1}{\abs{I}^3} \sup_{s \in \Gamma} \int_{\R \backslash I} \Abs{\frac{s(1-s)}{{s - \frac{1}{2} - \im \tau}}}^2 \dif \tau\right\}.
\end{equation}

The heart of the qualitative part of the statement is the following. 
\begin{quote}
\emph{Given $f \in H^2(\C_{1/2})$ and a bounded interval on the abscissa $\sigma = 1/2$, we can find a function $F \in \Hp^2$ such that $\Re F = \Re f$ on this interval.}
\end{quote}

The quantitative part of our main result gains additional interest as one observes, that it can be considered as a strong measure for how quickly the well-known almost periodicity for elements in $\Hp^2$ takes place in the vertical direction. The inequality \eqref{onto: corollary asymptotics}
is obtained by establishing a lower estimate for a frame consisting of weighted exponentials at logarithmic frequencies, see Lemma \ref{frame theorem}. This is of independent interest since explicit bounds for non-trivial frames are difficult to obtain.

The proof of Theorem \ref{main result}
 is given in Sections \ref{main result section} --\ref{proof of lemma construction}.

For the readers convenience we end this subsection by giving a simple and self-contained proof of the embedding \eqref{embedding} stated in the introduction. This proof  serves as a blueprint for the argument of our main result, and we also use directly some ensuing information from the proof. Note that this proof is different than the ones given in \cite{montgomery94,hls97}.
\begin{proof}[Proof of inequality \eqref{embedding}]
Let $\chi_I$ denote the indicator function of the interval $I \subset \R$ and consider the embedding operator defined on the set of Dirichlet polynomials $\mathscr{D}$ by
\begin{equation*}
	 E_I : \sum_{n \in \N} a_n n^{-s} \in \mathscr{D} \longmapsto \chi_I(t) \sum_{n\in \N} a_n n^{-1/2 - \im t}.
\end{equation*}
The operator $E_I$ is densely defined from $\Hp^2$ to $L^2(I)$. The lemma holds if and only if $E_I$ extends to a bounded operator on all of $\Hp^2$.
We approximate $E_I$ in the strong-operator topology by operators of the type
\begin{equation} \label{intro: approximate embedding}
	 E_{I,\delta} :  \sum_{n \in \N} a_n n^{-s} \in \mathscr{D} \longmapsto \chi_I (t) \sum_{n\in \N} a_n n^{-1/2 - \delta - \im t}.
\end{equation}
For fixed $\delta>0$, the operator $E_{I,\delta}:\Hp^2\to L^2(I)$ is well-defined and bounded 
since it corresponds to restriction on an interval on the line
$\sigma=1/2+\delta .$
It follows by a straight-forward computation that
\begin{equation*}
	 E_{I,\delta} E_{I,\delta}^\ast g = \sum_{n \in \N} \frac{\hat{g}(-\log n)n^{-\im t} }{n^{1+2\delta}},
\end{equation*}
where the Fourier transform on $L^2(\R)$ is given by
\begin{equation*}
	 \hat{g}(\xi) = \int_\R g(t) \e^{-\im t \xi} \dif t.
\end{equation*}
By expanding the Fourier transforms and interchanging summation and integral signs before taking the limit (which is easily justified by considering first $g\in C_0^\infty (I)$), we get
\begin{equation*}
	 E_{I,\delta} E_{I,\delta}^\ast g = \lim_{\delta \rightarrow 0} \chi_I (g \ast \zeta_{1 + 2\delta}).
\end{equation*}
where $\ast$ denotes convolution on $\R$, the function $\zeta_{\sigma}(t) = \zeta(\sigma + \im t) =  \sum n^{-\sigma - \im t}$ is  the Riemann-zeta function and $g \in L^2(I)$ is extended to all of $\R$ by setting it equal to zero outside of $I$. Since  the zeta function satisfies  $\zeta(s) = \frac{1}{s-1} + \psi(s)$ where $\psi$ is an entire function, 
we may take the limit $\delta\to 0^+$ to obtain
	\begin{equation} \label{formulaE}
		E_IE_I^\ast g = 2\pi \chi_I P_+g +  \chi_I (g \ast \psi_1).
	\end{equation}
Here $\psi_1(t) = \psi(1 + \im t)$ and  $P_+$ denotes the Riesz projection $L^2(\R) \rightarrow H^2(\C_+)$ given by
\begin{equation} \label{onto: riesz projection}
	 P_+g(t) = \lim_{\delta\to 0^+}  \frac{1}{2\pi} \int_\R  \frac{g(\tau)}{\delta+ \im (t - \tau)} \dif \tau.
\end{equation}
Since the Riesz projection is bounded on $L^2(\R)$, the lemma now follows. (See for instance \cite[p. 128]{koosis98} for more on the Riesz projection.)
\end{proof}

\subsection{The case $p=1$.}
Before we proceed, we try to better explain the structure of the spaces $\Hp^p$ for general $p \in [1,\infty)$ by indicating the connection to the Hardy spaces on the infinite polydisk $H^p(\T^\infty)$. The idea is to identify the Dirichlet monomial $p_n^{-s}$, where $p_n$ is the $n$'th prime number, with the $n$'th variable of the infinite torus 
\begin{equation*}
	\T^\infty = \Set{ (z_1, z_2, \ldots) : z_n \in \T }.
\end{equation*}
In this way we get the formal correspondence
\begin{equation*}
	 \mathscr{B} :   \sum_{n \in \N} a_n z^{\nu_1}_1 \cdots z_k^{\nu_k} \longmapsto \sum_{n \in \N} a_n n^{-s},
\end{equation*}
where $n = p_1^{\nu_1} \cdots p_k^{\nu_k}$ is the prime number decomposition of the integer $n$. This is the Bohr identification first introduced in \cite{bohr13a}.
Since $\T^\infty$ is a compact topological group, it is straight-forward to define the spaces $L^p(\T^\infty)$ and their subspaces $H^p(\T^\infty)$. For details, the reader should consult the references \cite{hls97,bayart02}. In particular, there it is shown, using ergodic theory, that for Dirichlet polynomials we have
\begin{equation*}
	 \lim_{T \rightarrow \infty} \frac{1}{2T} \int_{-T}^T \abs{D(\im t)}^p \dif t
	 =
	 \int_{\T^\infty} \abs{\mathscr{B}^{-1}D(\chi)}^p \dif \rho (\chi).
\end{equation*}
The reference \cite{saksman_seip08} contains a simpler proof of this relation, based on Weierstrass' density theorem.
This identity implies that the two norms for $\Hp^p$ are isometric and it follows that they are invariant under vertical translations. 
We next state a corollary to our main result. 
\begin{theorem} \label{corollary in introduction}
	Let $I \subset \R$ be a bounded  interval. Then for every $f \in H^1(\C_{1/2})$ there exist
	$F \in \Hp^1$ 
	and $\phi \in \mathrm{Hol}(\C_I)$ with $\Re \phi(1/2 + \im t) = 0$ on $I$ such that 
	$f = F + \phi$.
\end{theorem}
\begin{proof}
	We may assume that $I = [-T,T]$. Fix $\epsilon >0$ and 
	denote $I_\epsilon := [-T-\epsilon, T+\epsilon]$.
	For $f \in H^1(\C_{1/2})$ let  $f = JO$ be its unique factorisation into an inner function $J$ and an outer function $O$. Set $g = JO^{1/2}$
	and $h = O^{1/2}$. We have both $g, h \in H^2$, so  by Theorem \ref{main result} there exists functions $G, H \in \Hp^2$
	and $\phi_g,\phi_h \in \mathrm{Hol}(\C_{I_\epsilon})$ such that
	$g = G + \phi_g$ and $h = H + \phi_h$. Using this,
	\begin{equation*}
		f - GH =  g\phi_h +h\phi_g   - \phi_g \phi_h.
	\end{equation*}
	In particular $v = \chi_I \Re(f - GH) \in L^2(I)$. Let 
	\begin{equation*}
		\widetilde{v}(s) =  \frac{1}{2\pi \im } \int_I  v(\tau) \frac{1}{s - \im \tau - \frac{1}{2}} \dif \tau.
	\end{equation*}
	Then $\widetilde{v} \in H^2(\C_{1/2})$ and so we may find $V \in \Hp^2$ and $\phi_v \in \mathrm{Hol}(\C_I)$ such that $\widetilde{v} = V + \phi_v$.
	Let $F = GH + V$. It now follows that $F \in \Hp^1$ and moreover that $\Re(f - F)(1/2 + \im t) = 0$ on $I$.
\end{proof}
It is evident that one may use the quantitative bound of Theorem \ref{main result} to compute an explicit upper bound for $\norm{f-F}_{L^\infty(\Gamma)}$ in Theorem \ref{corollary in introduction} in the spirit of formula \eqref{main theorem quantitative bound}.

\subsection{General $p \in [1,\infty)$.} 
We refer to the previous subsection for a description of the spaces $\Hp^p$ in the general case. 
Our next results have a slightly different flavour than the previous ones. To motivate them,
we mention the role Dirichlet series play in number theory. In particular, the questions on the local embedding properties into $L^p$ spaces of intervals on the abscissa $\sigma=1/2$ appear to be in some sense analogues of certain deep
conjectures in analytic number theory, known as Montgomery's conjectures (see e.g. \cite[chapter 7]{montgomery94} and Bourgain \cite{bourgain00}).
For more information, we refer to
 \cite[section 3]{saksman_seip08}, which especially discusses in detail the open question that has become known as the Embedding problem for $\Hp^p$: given $p \in [1,\infty)$,  does there exist a constant $C>0$ such that 
\begin{equation} \label{p embedding}
		\lim_{\sigma \rightarrow \frac{1}{2}^+ }\int_I \Abs{F\left(\sigma + \im t \right)}^p  \dif t \leq C \norm{F}^p_{\Hp^p}
\end{equation}
holds?
Curiously enough,  this inequality is presently known only for
exponents $p \in 2\N$.

In Theorem \ref{equivalent conditions for embedding} we present two more equivalent statements which are in terms of Carleson measures. Before we state it, we give some additional definitions and prove a preliminary result, Theorem \ref{carleson measure lemma}.
Let $\mu$ be a positive measure defined on $\C_{1/2}$. 
We say that $\mu$ is a Carleson measure for the space $H^p(\C_{1/2})$ if there exists a constant $C>0$ such that for every
$f \in H^p(\C_{1/2})$ we have
\begin{equation} \label{definition: carleson measure}
	 \int_{\C_{1/2}} \abs{f(s)}^p \dif \mu(s) \leq C \norm{f}^p_{H^p(\C_{1/2})}.
\end{equation}
We call a measure with bounded support a local measure.
Following \cite{olsen_seip08}, we say that a positive measure $\mu$ is a Carleson measure for the space $\Hp^p$ if there exists a constant $C>0$ such that for all $F \in \Hp^p$ it holds that
\begin{equation*}
	 \int_{\C_{1/2}} \abs{F(s)}^p \dif \mu(s) \leq C \norm{F}^p_{\Hp^p}.
\end{equation*}
With this definition, the embedding \eqref{embedding} implies that any Carleson measure
for the space $H^2(\C_{1/2})$ with bounded support is also a Carleson measure for the 
space $\Hp^2$. By a straight-forward argument, we are able to establish the following result.
\begin{theorem}\label{carleson measure lemma}
	Let $p \in [1,\infty)$ and assume that $\mu$ is a positive measure on $\C_{1/2}$. If $\mu$ is a Carleson measure for $\Hp^p$  then $\mu$ is a Carleson measure for $H^p(\C_{1/2})$.
\end{theorem}

The proof relies on Carleson's geometric characterisation of the Carleson measures \cite{carleson62} for the spaces $H^p(\C_{1/2})$. 
We call a subset $Q \subset \C_{1/2}$ a Carleson square if it is a square with one side on the abscissa $\sigma = 1/2$.
\begin{lemma}[Carleson]  \label{carleson geometric characterisation}
	Let $p>0$. A positive measure $\mu$ on $\C_{1/2}$ is a Carleson measure for the space $H^p(\C_{1/2})$ if and only if there is a constant $C>0$ such that for all Carleson squares $Q$ we have $\mu(Q) \leq C \abs{Q}$.
\end{lemma}
We make an important remark. The constant $C$ of Carleson's characterisation is comparable to the best constant of the inequality we used to define the Carleson measures.

\begin{proof}[Proof of Theorem \ref{carleson measure lemma}]
	Assume that $\mu$ is a Carleson measure for $\Hp^p$ with constant $C>0$. 
	Let $Q$ be a small Carleson box in $\C_{1/2}$. 
	Let $s_0$ be the mid-point of the right edge of the box. 
	Next, we verify that $\zeta^{2/p}_{s_0}(s) = \zeta^{2/p}(s + \overline{s_0}) \in \Hp^p$
	and compute the norm.
	Consider the function
	\begin{equation*}
	 	F(z_1, \cdots) = \prod \left( \frac{1}{1 - p_n^{-\overline{s_0}}z_n} \right)^{2/p}
	\end{equation*}
	on the infinite dimensional torus. Both checking its norm and computing its Fourier coefficients is straight-forward since the evaluation against the measure Haar measure on $\T^\infty$ splits over the coordinates. In particular,
	\begin{equation*}
		\norm{F}_{L^p(\T^\infty)}^p
		=
	 	\prod_{n \in \N} \Norm{\left( \frac{1}{1 - p_n^{-\overline{s_0}}z_n} \right)^{2/p}}_{L^p(\T)}^p 
		= \prod_{n \in \N} \left( \frac{1}{1 - p_n^{-2\sigma_0}} \right) = \zeta(2 \sigma_0).
	\end{equation*}
	In this way it follows that $F \in H^p(\T^\infty)$, and so by the Bohr identification, we obtain $\mathscr{B} F = \zeta^{2/p}_{s_0} \in \Hp^p$ with $\norm{\zeta_{s_0}^{2/p}}_{\Hp^p}^p = \norm{\zeta_{s_0}}_{\Hp^2}^2 =  \zeta(2\sigma_0)$.

	Next, we combine this with the fact that $\mu$ is a Carleson measure for $\Hp^p$ to get
	\begin{equation} \label{carleson measure lemma equation 1}
		\int_Q \frac{\abs{\zeta_{s_0}^{2/p}(s)}^p}{\norm{\zeta_{s_0}}^2_{\Hp^2}} \dif \mu
		\leq C \frac{1}{\norm{\zeta_{s_0}}^2_{\Hp^2}} \norm{\zeta_{s_0}^{2/p}}_{\Hp^p}^p = C.
	\end{equation}
	On the other hand, by the formula $\zeta(s) = (s-1)^{1} + \psi(s)$ for the Riemann zeta function, where $\psi$ is an entire function, it follows that 
	$\zeta(2\sigma_0)^{-1}= \norm{\zeta_{s_0}}^{-2}_2  =  ( 2\sigma_0-1)(1 + \Littleoh{1})$ as $\sigma_0 \rightarrow 1/2$. So, for $s_0$ close to the abscissa $\sigma= 1/2$, the left hand side of (\ref{carleson measure lemma equation 1}) is greater than some constant times
	\begin{equation} \label{inequality abcd}
		(2\sigma_0-1)\int_Q \Abs{\zeta(s + \overline{s_0})}^2\dif \mu 
		\geq
		\frac{2\sigma_0-1}{2} \int_Q \left(  \Abs{\frac{1}{s + \overline{s_0} -1} }^2 - 2\Abs{\psi(s+\sigma_0)}^2\right)\dif \mu.
	\end{equation}
	Since $\psi(s) = \Bigoh{1}$ for $s$ in a bounded set, it follows by geometric considerations that 
	$\abs{s + \overline{s_0} -1}^2 \leq (5/4) (1 - 2\sigma_0)^2$ for $s \in Q$. Hence, the expressions in (\ref{inequality abcd}) are greater than
$({2}/{5}) ({2\sigma_0-1})^{-1} \mu(Q) + \Bigoh{2\sigma_0-1}$.
	It follows that there is some constant $D>0$  such that for 
	any Carleson box with $\sigma_0 < 1$ we have
	\begin{equation*}
		\mu(Q) \leq  D( 2 \sigma_0-1).
	\end{equation*}
	
	We verify that this implies that $\mu$ is a Carleson measure for $\Hp^p$.
	Since $1 \in \Hp^p$, it follows that $\mu(\C_{1/2}) \leq C$. So for any Carleson box $Q_2$ with sides   $\sigma_0 \geq 1$, it follows that
$\mu(Q_2) \leq C \leq  {C}  (2\sigma_0 - 1)$.
	Hence, by Carleson's characterisation of Carleson measures (Lemma \ref{carleson geometric characterisation}), $\mu$ is a Carleson measure 
	for the spaces $H^p(\C_{1/2})$.
\end{proof}		
The next theorem says that the converse of this theorem is equivalent to the truth of the Embedding conjecture, i.e. the validity of the inequality \eqref{p embedding} for general $p \in [1,\infty)$.
\begin{theorem} \label{equivalent conditions for embedding}
	Let $p \in [1,\infty)$. Then the following statements are equivalent.
	\begin{itemize}
		\item[(a)] For every bounded interval $I\subset \R$ there exists a constant $C = C(\abs{I})>0$ such that for all finite sequences $(a_n)$ of complex numbers
	it holds that
	\begin{equation*}
		 \int_I \Abs{\sum a_n n^{-\frac{1}{2} - \im t}}^p \dif t \leq C \Norm{\sum a_n n^{-s}}^p_{\Hp^p}.
	\end{equation*}
	\item[(b)] Every local Carleson measure for $H^p(\C_{1/2})$ is also a Carleson measure
	for $\Hp^p$.
	\item[(c)] There exists a constant $D>0$ such that every local Carleson measure for $H^p(\C_{1/2})$ of the form
	\begin{equation*}
		\mu_S = \sum \delta_{s_n} (2 \sigma_n - 1),
	\end{equation*}
	with constant $C>0$,
	is also a Carleson measure for $\Hp^p$ with
	\begin{equation*}
		\int \abs{F(s)}^p \dif \mu_S(s) \leq  CD \norm{F}^p_{\Hp^p} \quad \forall F \in \Hp^p.
	\end{equation*}
	\end{itemize}
\end{theorem}

We give a lemma on which the implication $(b) \Rightarrow (c)$ hinges.
To do this we need to make some definitions.
	Let $\Gamma$
	be some bounded subset of $\C_{1/2}$ and let $M(\Gamma)$ denote the complex
	measures with support in the closure of $\Gamma$. This forms a Banach space
	under the norm $\norm{\mu} = \int_\Gamma \dif \abs{\mu}$.
	For fixed $p \in [1,\infty)$ let $X$ denote either of the spaces $H^p(\C_{1/2})$ or $\Hp^p$.
	By
	 $\mathrm{CM}^p_\Gamma(X)$
	 we denote the space of all signed measures supported on a bounded subset $\Gamma \subset \C_{1/2}$ 
	equipped
	with the norm
	\begin{equation*}
		\norm{\mu}_{\mathrm{CM}^p_\Gamma(X)} = \sup_{\norm{f}_X = 1} \int_\Gamma \abs{f(s)}^p \dif \abs{\mu(s)}.
	\end{equation*}
	Here $\nu = \abs{\mu}$ denotes the total variation measure of $\mu$ (see \cite[p. 93]{folland99}).
\begin{lemma} \label{carleson banach}
	For fixed $p \in [1,\infty)$, let $X$ denote either of the spaces $H^p(\C_{1/2})$ or $\Hp^p$.
	Then
	the space $\mathrm{CM}^p_\Gamma(X)$ 
	is a Banach space.
\end{lemma}
\begin{proof}
	Since $1 \in \Hp^p$ and $s^{-2} \in H^p$, it
	follows that
	\begin{equation} \label{carleson measure norm and total variation}
		\norm{\mu_n}_{M(\Gamma)} = \int_\Gamma \dif \abs{\mu_n} \lesssim \norm{\mu_n}_{CM^p_\Gamma(X) }.
	\end{equation}
 	Assume that $\seq{\mu_n} \subset \mathrm{CM}_\Gamma^p(X)$ is such that $\sum \norm{\mu_n}_{CM^p_\Gamma(X)} < \infty$.
	It suffices to show that $\sum \mu_n$ is convergent in $CM^p_\Gamma(X)$. By the inequality (\ref{carleson measure norm and total variation}), we have
	\begin{equation*}
		\sum \norm{\mu_n}_{M(\Gamma)} \leq \sum \norm{\mu_n}_{CM^p_\Gamma(X) } < + \infty,
	\end{equation*}
	and so $\sum \mu_n$ converges to some element $\mu \in M(\Gamma)$. Moreover,
	$\mu \in CM^p_\Gamma(X)$. Indeed, for any polynomial $D$, 
	\begin{equation*}
		 \int \abs{D(s)}^p \dif \abs{\mu} \leq \norm{D}_{X}^p \sum \norm{\mu_n}_{CM^p_\Gamma(X)}.
	\end{equation*}
	Finally, we confirm that $\mu_n \rightarrow \mu$ in the sense of $CM^p_\Gamma(X)$. But this
	follows immediately, since $\norm{ \mu - \sum_{n=1}^N \mu_n } \leq \sum_{n > N} \norm{\mu_n}_{CM^p_\Gamma(X)}$. In conclusion, $CM^p_\Gamma(X)$ is a Banach space.
\end{proof}

\begin{proof}[Proof of Theorem \ref{equivalent conditions for embedding}]
	It is clear that $(a) \Rightarrow (b)$ since then $(1/s) \Hp^p \subset H^p$ (compare to formula \eqref{introduction trick}). We proceed to show $(b) \Rightarrow (c)$
	and $(c) \Rightarrow (a)$.

	$(b) \Rightarrow (c)$: 
	Let $\Gamma$ be some bounded subset of $\C_{1/2}$.
	Consider the operator
	\begin{equation*}
		 \mathcal{I} : \mu \in CM^p_\Gamma(H^p) \longmapsto \mu \in CM^p_\Gamma(\Hp^p).
	\end{equation*}
	By the hypothesis and Lemma \ref{carleson banach}, the operator  $\mathcal{I}$ is well-defined. It suffices to show that it is continuous.
	By the closed graph theorem this follows if it has a closed graph. 
	Assume that $\mu_n \rightarrow \mu$ in $CM^p_\Gamma(H^p)$ and that $\mu_n \rightarrow \nu$ in $CM^p_\Gamma(\Hp^p)$.
	By (\ref{carleson measure norm and total variation}), this implies that both $\mu_n \rightarrow \mu$ and
	$\mu_n \rightarrow \nu$ in the topology of $M(\Gamma)$, and so $\mu = \nu$ as measures. Hence $\mathcal{I}$ has a closed graph. Finally, $(c)$ is just a special case of boundedness of $\mathcal{I}$ applied to sums of the point masses $\delta_{s_n}$.
	
	$(c) \Rightarrow (a)$:
	Let $F \in \Hp^p$ be a Dirichlet polynomial and consider 
	\begin{equation*}
		\int_0^T \Abs{F\left( \frac{1}{2} + \epsilon + \im t \right)}^p  \dif t.
	\end{equation*}
	For $\epsilon > 0$ small enough, the above is less than
	\begin{align*}
		 \sum_{n=0}^{T[\epsilon^{-1}]} \int_{\epsilon n}^{\epsilon(n+1)} 
		 \Abs{F \left( \frac{1}{2} + \epsilon + \im t \right)}^p \dif t
		 &=
		 \sum_{n=0}^{T[\epsilon^{-1}]} \int_0^\epsilon \Abs{F\left( \frac{1}{2} + \epsilon + \im t + \im n \epsilon \right)}^p \dif t\\
		 &= 
		 \frac{1}{\epsilon} \int_0^\epsilon \int_\C \Abs{F(s)}^p \dif \mu_{\epsilon,t}(s) \dif t,
	\end{align*}
	where
		$ \mu_{\epsilon,t} = \epsilon \sum_{n=0}^{T[\epsilon^{-1}]} \delta_{\frac{1}{2} + \epsilon + \im t + \im n \epsilon } $.
	By Carleson's geometric characterisation of Carleson measures (Lemma \ref{carleson geometric characterisation}),
	the quantities $\norm{\mu_{\epsilon,t}}_{\mathrm{CM}^p(H^p)}$ are uniformly bounded for $\epsilon \in (0,1)$. 
	 Let $\Gamma \subset \C_{1/2}$ be a bounded subset of $\C_{1/2}$ such that the supports of the measures $\mu_{\epsilon, t}$ for $\epsilon \in (0,1)$ are contained in $\Gamma$. Then the uniform boundedness also holds in the  norm of 
	$\mathrm{CM}^p_\Gamma(H^p)$.  By $(c)$, this implies that for $\epsilon \in (0,1)$ we have
	\begin{equation*}
		\int_\C \Abs{F(s)}^p \dif \mu_{\epsilon,t}(s) 
		 \lesssim \norm{F}^p_{\Hp^p}.
	\end{equation*}
	Hence,
	\begin{equation*}
		\int_0^T \Abs{F\left( \frac{1}{2} + \epsilon + \im t \right)}^p  \dif t \lesssim  \norm{F}^p_{\Hp^p}
	\end{equation*}
	as $\epsilon \rightarrow 0$, and the embedding theorem holds for $\Hp^p$.
\end{proof}

\subsection{McCarthy's spaces}
In \cite{mccarthy04} J.~E.~McCarthy studied several Hilbert spaces of Dirichlet series. In particular, he defined the family of spaces
\begin{equation*}
	 \Hp^2_\alpha = \Set{ \sum_{n\in \N} a_n n^{-s} : \sum_{n\in \N} \abs{a_n}^2 \log^\alpha (n+1) < \infty}, \alpha \in \R.
\end{equation*}
By the Cauchy-Schwarz inequality the elements of these spaces are analytic on $\C_{1/2}$.
The motivation for introducing these spaces is that they resemble the classical scale of spaces $D_\alpha(\D)$ that includes the Bergman, Hardy and Dirichlet spaces on the unit disc. Specifically, for $\alpha<0$, it is the weighted Bergman space of
functions analytic in $\C_{1/2}$ such  that
\begin{equation*}
	\norm{f}_{D_\alpha}^2= \int_{\C_{1/2}} |f(s)|^2
\left(\sigma-\frac12\right)^{-\alpha-1} \dif m(s)< + \infty.
\end{equation*}
For $0<\alpha<2$ we let $D_\alpha(\C_{1/2})$ be the Dirichlet-type of
space of functions analytic in $\C_{1/2}$ such that $f(\sigma)\to 0$
when $\sigma\to\infty$ and
\begin{equation*}
	\norm{f}_{D_\alpha}^2= \int_{\C_{1/2}} |f'(s)|^2
\left(\sigma-\frac12\right)^{-\alpha+1} \dif m(s)< + \infty.
\end{equation*}
For $\alpha\geq 2$ it is possible to define these spaces using higher order derivatives.
In \cite{olsen_seip08}, embeddings analogous to \eqref{embedding} were found and used to determine their bounded interpolating sequences. To state these embeddings we let $Q_\theta$ denote
the half-strip $\sigma>1/2$, $\theta<t<\theta+1$. 
\begin{lemma} \label{onto: sobolev embedding lemma}
  Let $F \in \Hp^2_\alpha$. 
  For $\alpha<0$, we have
  \begin{equation} \label{bergman}
    \int_{Q_{\theta}} \abs{F(s)}^2 \left(
  \sigma-\frac{1}{2}\right)^{-\alpha -1} \dif m(s)
    \leq C \norm{F}_{\Hp_\alpha}^2.
  \end{equation}
  Similarly, it holds for $0<\alpha < 2$ that
  \begin{equation} \label{dirichlet}
    \int_{Q_{\theta}} \abs{F'(s)}^2 \left(
  \sigma-\frac{1}{2}\right)^{-\alpha +1} \dif m(s)
    \leq C \norm{F}_{\Hp_\alpha}^2.
  \end{equation}
\end{lemma}
In particular, by using the same trick as for the space $\Hp^2$ it follows that $F \in \Hp^2_\alpha$ implies that $F/s \in D_\alpha(\D)$.
Our result in this setting is the following.
\begin{theorem} \label{sobolev matching real parts}
	Let $I\subset \R$ be a bounded and open interval. 
	Then for every $f \in D_\alpha(\C_{1/2})$ there exists an
	$F \in \Hp^2_\alpha$ such that $f- F$ has continues analytically to all of $\mathrm{Hol}(\C_I)$
	with $\Re(f-F)(1/2 + \im t) = 0$ on $I$.
	There exists a unique $F \in \Hp^2_\alpha$ of minimal norm satisfying this. Moreover, there exists a constant $C  >0$, depending only on $\alpha \in \R$ and the length of $I$, such that
	\begin{equation*}
	 	\norm{F}_{\Hp^2_\alpha} \leq C  \norm{f}_{D_\alpha}^2.
	\end{equation*}
\end{theorem}
We establish this result by considering the operator defined on finite sequences by
\begin{equation*}
	 R_I : (a_n)_{n \in \Z^\ast} \longmapsto  \left( \sum_{n \in \N} \frac{a_n n^{-\im t} + a_{-n}n^{\im t}}{\sqrt{n} } \right) \Bigg|_I.
\end{equation*}
In order to determine the proper domain and target spaces for this operator we need to introduce Sobolev spaces that in general consist of distributions.
Note that since  multiplying distributions with the indicator function is in general problematic, we consider restrictions instead.
The proof is given in section \ref{section: mccarthy spaces} along with the needed definitions.

\section{A brief introduction to frames} \label{section: preliminaries}
A sequence of vectors $\seq{f_n}$ for a Hilbert space $H$ is called a frame for $H$, with a lower frame bound $A$ and an upper frame bound $B$, if
\begin{equation} \label{onto: frame definition}
	A\norm{g}^2 \leq \sum \abs{\inner{g}{f_n}}^2 \leq B\norm{g}^2 \quad \forall g \in H.
\end{equation}
The optimal constants in the inequalities are called the upper and lower bounds for the frame in question.
	A frame may be seen as a sort of over-complete basis. 
	In fact, a sequence $\seq{f_n}$ is a frame 
	for a Hilbert space $H$ if and only if for any element $f \in H$
	there exists a sequence $a_n \in \ell^2$ such that $f = \sum a_n f_n$ 
	with $\norm{f}^2_H \sim \sum \abs{a_n}^2$, and for any sequence $(a_n) \in \ell^2$ the sum $\sum a_n f_n$ converges in $H$. However, the choice
	of the coefficients is  in general not unique. If it is, we call $\set{f_n}$
	a Riesz basis for $H$.
Assume that
\begin{equation} \label{frame operator}
	\mathcal{S} : (a_n) \in \ell^2 \longmapsto \sum a_n f_n \in H
\end{equation}
defines a bounded operator.
It is not hard to see that the adjoint operator $\mathcal{S}^\ast$ is bounded
below  if and only if $\seq{f_n}$ 
is a frame for $H$. Given this fact, the following lemma in reality gives
equivalent conditions for being a frame.
The reader may wish to consult \cite[p. 97]{rudin73} for the equivalence of $(a)$ and $(b)$.
\begin{lemma} \label{frame lemma}
	Let $X,Y$ be Hilbert spaces and $A : X \rightarrow Y$ a bounded operator. Then the following
	are equivalent.
	\begin{itemize}
		\item[(a)] $A$ is onto.
		\item[(b)] $A^\ast$ is bounded  below.
		\item[(c)] $AA^\ast$ is bounded  below.	 
	\end{itemize}\end{lemma}
	In the literature, the operator $\mathcal{S} \mathcal{S}^\ast$ 
is often called the frame operator. Note that since $\mathcal{S} \mathcal{S}^\ast$ 
is self-adjoint, then it is bounded below if and only if it is onto. 
For a frame $\seq{f_n}$, the sequence $\seq{(\mathcal{S}\mathcal{S}^\ast)^{-1} f_n}$ 
is also a frame, and it is called the canonical dual frame.
The following is shown in, e.g., \cite[ch. 5]{christensen03}.
\begin{lemma} \label{frame lemma II}
	Let $\seq{f_n}$ be a frame in a Hilbert space $H$ that has frame bounds $A,B> 0$, as in \eqref{onto: frame definition},
	and let $\seq{g_n}$ denote its canonical dual frame. Then for all $f \in H$ 
	we have 
	\begin{equation*}
		\frac{1}{B} \norm{f}^2 \leq \sum \abs{\inner{f}{g_n} }^2 \leq \frac{1}{A} \norm{f}^2,
	\end{equation*}
	and the representation
	\begin{equation*}
		f = \sum \inner{ f}{g_n} f_n.
	\end{equation*}
	Moreover, if $(a_n)$ is a sequence such that $f = \sum a_n f_n$, then 
	\begin{equation*}
		\sum \abs{\inner{ f}{g_n} }^2
		\leq
		\sum \abs{a_n}^2. 
	\end{equation*}
\end{lemma}
We note that if $\seq{f_n}$ is a Riesz basis, then the canonical dual frame is 
also a Riesz basis and it is uniquely defined through the bi-orthogonal relation $\langle f_m | g_n \rangle = \delta_{m,n}$.

\section{Proof of Theorem \ref{main result}} \label{main result section}
\subsection{Qualitative part}
	The following lemma is shown in e.g. \cite[p. 238, thm. 5.26]{kato66}.
	\begin{lemma}[Second stability theorem of Semi-Fredholm theory] \label{second stability theorem}
    	Let $X,Y$ be Banach spaces and $A : X \rightarrow Y$ be
		a continuous linear operator that is bounded below. If $K : X \rightarrow Y$
		is a compact operator and $B = A + K$ is injective, then it follows 
		that $B$ is bounded below.
	\end{lemma}

Recall that $\Z^\ast = \Z \backslash \{ 0 \}$.
 Let $I \subset \R $ be  a bounded interval and  define the operator
\begin{equation*}
	 R_I : \seq{a_n}_{n \in \Z^\ast} \in \ell^2(\Z^\ast)
	 \longmapsto \chi_I(t)
	 \sum_{n \in \N} \frac{a_n n^{-\im t} + a_{-n}n^{\im t}}{\sqrt{n}}
	 \in L^2(I).
\end{equation*}
The operator is well defined and bounded. Indeed, by writing
$f_1=\sum_{n=1}^\infty a_n n^{-s}$ and $f_2=\sum_{n=1}^\infty \overline{a_{-n}} n^{s}$, we see that $R_I (\seq{a_n}_{n \in \Z^\ast})=f_1(1/2+\im t)+\overline{f_2 (1/2+\im t)}$, and the boundedness
follows from that of the embedding operator $E_I$.
The following lemma is the analogue of the identity \eqref{formulaE},
and the proof is completely analoguous to that of our proof of
the embedding theorem in Subsection \ref{case2}, whence we
just indicate the needed changes.
\begin{lemma}
	Let $R_I : \ell^2(\Z^\ast) \rightarrow L^2(I)$ be the operator defined above. Then
	\begin{equation} \label{harmonic embedding as convolution formula}
		R_I R_I^\ast g = 2\pi g + \chi_I (g \ast \phi),
	\end{equation}
	where $\phi(t) = \Re \psi(1+\im t)$, and $\psi$ is an entire function.
\end{lemma}
\begin{proof}
We approximate $R_I$ in the strong-operator topology by operators of the type $R_{I,\delta}, $ where for $\delta >0$
\begin{equation} \label{end: approximate embedding}
	 R_{I,\delta} : \seq{a_n}_{n \in \Z^\ast} \in \ell^2(\Z^\ast)
	 \longmapsto \chi_I(t)
	 \sum_{n \in \N} \frac{a_n n^{- \im t} + a_{-n}n^{ + \im t}}{n^{1/2 + \delta}}
	 \in L^2(I).
\end{equation}
This time  a straight-forward computation yields that
\begin{equation*}\
	 R_I R_I^\ast g = \lim_{\delta \rightarrow 0} \chi_I (g \ast 2\Re \zeta_{1 + 2\delta}).
\end{equation*}
By invoking again  $\zeta(s) = \frac{1}{s-1} + \psi(s)$,
and observing that $\pi^{-1}\Re (\delta+it)^{-1}=\pi^{-1}\delta/(\delta^2+t^ 2)$ equals
the Poisson kernel for the half plane,
we have established that
	\begin{equation} \label{formula: embedding operator}
		R_I R_I^\ast g = 2\pi g +  \chi_I (g \ast \phi).
	\end{equation}
Here $\phi(t) = \Re \psi(1 + \im t)$.
\end{proof}
	We combine the previous two lemmas to obtain the following.
	\begin{lemma} \label{RI is onto}
		The operator $R_I$ is onto $L^2(I)$.
	\end{lemma} 
	\begin{proof}
	We may assume that $I$ is centered at the origin.
	By Lemma \ref{frame lemma} it suffices to check that 
	$R_I R_I^\ast$ is bounded below in norm.
	We consider the formula (\ref{harmonic embedding as convolution formula}). 
	The first term on the right hand side is a constant multiple of the identity operator which is bounded below. 
	In order to use Lemma \ref{second stability theorem} to show that this implies that $R_I R_I^\ast$ is bounded below, it suffices to show that the operator $R_I R_I^\ast$ is injective, 
	and that the operator $\Phi: g \longmapsto \chi_I (g \ast \phi)$ is compact.
	The last assertion follows immediately from the boundedness of the kernel of the operator, which makes it Hilbert-Schmidt.
	
	We show that $R_I R_I^\ast$ is injective.
	Since $R_I$ is always injective on the image of $R_I^\ast$, it suffices to check that $R_I^\ast$ is injective. A simple calculation shows that
	\begin{equation*}
		R^\ast_I g = \left( \ldots, \; \frac{\hat{g}(\log 3)}{\sqrt{3}}, \;  \frac{\hat{g}(\log 2)}{\sqrt{2}},  \; \hat{g}(0), \;  \frac{\hat{g}(-\log 2)}{\sqrt{2}}, \frac{\hat{g}(-\log 3)}{\sqrt{3}}, \ldots \right). 
	\end{equation*}
	Hence, we need to check that for $g \in L^2(I)$ the condition $\hat{g}(\pm \log n) = 0$ for all $n \in \N$ implies $g = 0$.
	Assuming that $g$ is non-zero we obtain a contradiction, since it is well-known (see e.g. \cite[Theorem 3, p. 61]{young01}) that the density of zeros of functions of exponential type is much less than that of our logarithmic sequence. 
\end{proof}
With these preparations out of the way, we are ready to give the proof of the qualitative part of Theorem \ref{main result}.
 Let $f \in H^2(\C_{1/2})$ and $I = (-T,T)$. Set $v = \chi_I \Re f(1/2 + \im t)$. 
	Since $R_I$ is surjective and $v \in L^2(I)$, there exists a sequence $(\gamma_n)_{n \in \Z^\ast} \in \ell^2(\Z^\ast)$ such that 
	\begin{equation*}
		v = \sum_{n \in \N} (\gamma_n n^{-1/2 - \im t} + \gamma_{-n} n^{-1/2 + \im t}), 
	\end{equation*}
	with the convergence of the sum being in the sense of $L^2(I)$.
	Since $v=(v+\overline{v})/2$, we may assume that $\gamma_{-n}=\overline{\gamma_n}.$
	It follows that the function
	\begin{equation*}
	F(s) = 2\sum_{n \in \N} \gamma_n  n^{-s} 
	\end{equation*}
	is in $\Hp^2$ and satisfies $\Re F(1/2 + \im t) = v(t)$ in the sense of $L^2(I)$.
	Hence, the function $f - F$ has vanishing real parts on $I$. 
	
	We next apply Schwarz reflection type argument via a suitable reproducing formula for $(F-f)$ to establish the existence of the coefficients $C_I$ and $D_{\Gamma,I}$.
	For $g \in H^2(\C_{1/2})$, it is well-known that a representation formula
	based on the real boundary values holds,
	\begin{equation*}
		g(s) = \frac{1}{\pi \im} \int_{\frac{1}{2} - \im \infty}^{\frac{1}{2} + \im \infty} \frac{\Re g(w)}{s-w} \dif w.		
	\end{equation*}
	As $s(1-s)$ is real on the abscissa $\sigma=1/2$, we apply the above formula to the function ${\big((f-F)(s) - (f-F)(1)\big)}/{s(1-s)}$
	to obtain the representation formula
	\begin{equation} \label{onto: schwarz reflection}
		(f-F)(s) \\= \frac{s(1-s)}{\im \pi} \int_{\frac{1}{2} - \im \infty}^{\frac{1}{2} + \im \infty}
		\frac{\Re\Big((f-F)(w) - (f-F)(1)\Big)}{w(1-w)(s-w)} \dif w + (f-F)(1). 
	\end{equation}
	Since $\Re(f-F)(1/2 + \im t)$ vanishes almost everywhere for $t \in I$, it follows that the integral defines an analytic function for all $s \in \C_I$. We denote this extension by $f-F$, and we note that it satisfies
	\begin{equation*}
		 	(f-F)(s) = - \overline{(f-F)(1 - \overline{s})} \quad s \in \C. 
	\end{equation*}

	We turn to statement (1). Since $R_I$ is onto and bounded, it follows by the open mapping theorem that there exists a constant $C>0$ such that the sequence $\seq{\gamma_n}$ above may be chosen to satisfy $\norm{\seq{\gamma_n}}_{\ell^2}^2 \leq C \norm{v}_{L^2(I)}^2$. Define $F$ in the same way as above.
	Since $\norm{v}_{L^2(I)} \leq \norm{f}_{H^2(\C_{1/2})}$, it follows that 
	\begin{equation*}
		\norm{F}_{\Hp^2}^2 \leq 2 \norm{\gamma_n}_{\ell^2}^2  \leq 2 C\norm{f}_{L^2(I)}^2
		\leq  2 C\norm{f}_{H^2(\C_{1/2})}^2.
	\end{equation*}
	The existence and uniqueness of the element of minimal norm follows from the general theory of convex sets: a closed convex set of 
	a Hilbert space always has a unique element of minimal norm (see, for instance,
	section 2.2 of \cite{duren_schuster04}). In particular, the set 
	\begin{equation*}
		\set{ F \in \Hp^2 :  \Re F(1/2 + \im t) = \Re f(1/2 + \im t) \; \text{as functions in} \; L^2(I)} 
	\end{equation*}
	is clearly closed and convex in $\Hp^2$.

	Statement (2) is an easy consequence of being able to choose the element $F$ prescribed by statement (1), and then applying the representation \eqref{onto: schwarz reflection}.

\subsection{Quantitative part}

In this section, we consider a dual formulation of Lemma \ref{RI is onto}. 
This will help us determine the asymptotic behaviour of the constant 
$C_I$ of Theorem \ref{main result}. 
The observation is that $R_I R_I^\ast$ is the frame operator in the space $L^2(I)$ for the sequence
\begin{equation*}
		\Omega_I^\ast = \left( \ldots, \; \frac{3^{\im t}}{\sqrt{3}}, \; \frac{2^{\im t}}{\sqrt{2}}, \; 1, \; 1, \; \frac{2^{-\im t}}{\sqrt{2}}, \; \frac{3^{-\im t}}{\sqrt{3}}, \; \ldots \right).
\end{equation*}
Lemma \ref{RI is onto} says that $R_I$ is onto, which by Lemma \ref{frame lemma} establishes that $\Omega_I^\ast$ really is a frame. Note, however, that the operator $R_I$ counts the constant function of $L^2(I)$ twice. 
We find it convenient to only count this function once. Therefore we define
\begin{equation*}
	 \mathcal{S}_I : \seq{a_n}_{n \in \Z \backslash \{ -1, 0 \} } \in \ell^2(\Z \backslash \{ -1,0\})
	 \longmapsto 
	 \chi_I(t) \left( a_1  + 
	 \sum_{n \geq 2} \frac{a_n n^{-\im t} + a_{-n}n^{\im t}}{\sqrt{n}} \right)
	 \in L^2(I).
\end{equation*}
Then $\mathcal{S}_I \mathcal{S}_I^\ast$ is the frame operator for
\begin{equation*}
		\Omega_I = \left( \ldots, \; \frac{3^{\im t}}{\sqrt{3}},  \; \frac{2^{\im t}}{\sqrt{2}}, \; 1, \; \frac{2^{-\im t}}{\sqrt{2}}, \; \frac{3^{-\im t}}{\sqrt{3}}, \; \ldots \right).
\end{equation*}
We prove the following lemma.
\begin{lemma} \label{frame theorem}
	Let $I \subset \R$ be an interval. Then $\Omega_I$
	forms a frame for $L^2(I)$. I.e., there exists constants $A_{I}, B_{I} > 0$, depending only on the length of $I$, such 
	that for all $f \in L^2(I)$ we have
	\begin{equation*}
		A_{I} \norm{f}^2 \leq \abs{\hat{f}(0)}^2 + \sum_{n \geq 2} \frac{\abs{\hat{f}(\log n)}^2 + \abs{\hat{f}(-\log n)}^2}{n}  \leq B_{I} \norm{f}^2.
	\end{equation*}
	Moreover, given $\epsilon >0$, the biggest choice for $A_I$ satisfies, for large enough $\abs{I}$,
	\begin{equation} \label{asymptotic lower bound}
	 	 \abs{I}^{-(1+\epsilon)\frac{6\abs{I}}{\pi}\log 2} \leq A_{I} \leq  \abs{I}^{-(1-\epsilon)\frac{I}{\pi} \log \pi}.
	\end{equation}
	Also, the smallest choice for $B_I$ satisfies
	\begin{equation} \label{upper bound inequalities}
		 \abs{I} \leq B_{I} \leq \abs{I} + d,
	\end{equation}
	where $d \leq 13.3138\dots$.
	Finally,
	\begin{equation} \label{asymptotics for small I}
		\lim_{\abs{I} \rightarrow 0} A_{I} = \lim_{\abs{I}\rightarrow 0} B_{I} =  2 \pi.
	\end{equation}
\end{lemma}
The complete proof of this lemma is rather long, so we immediately show how the quantitative part of Theorem \ref{main result} is a consequence of this result.
\begin{proof}[Proof of the remaining part of Theorem \ref{main result}]
	Let $f \in H^2(\C_{1/2})$ and $I = (-T,T)$. 	Set $v = \chi_I \Re f(1/2 + \im t)$. 
	Since $R_I R_I^\ast$ is the frame operator for the sequence $\Omega^\ast_I$, it follows by Lemma \ref{frame lemma II} that for the choice $\gamma_n := \inner{v}{(R_I R_I^\ast)^{-1} n^{-1/2 - \im t}}$ for $n\geq 1$, and
	correspondingly for $n \leq -1$, we have
	\begin{equation*}
		v = \sum_{n \in \N} (\gamma_n n^{-1/2 - \im t} + \gamma_{-n} n^{-1/2 + \im t}).
	\end{equation*}
	As in the proof of the qualitative part of the theorem, we find that we may assume $\gamma_n = \bar{\gamma}_{-n}$ and the function
	\begin{equation*}
		F(s) = 2\sum_{n \in \N} \gamma_n  n^{-s}
	\end{equation*}
	satisfies $\Re  F(1/2 + \im t) = v(t)$ in the $L^2(I)$ sense, and so
	\begin{equation} \label{eqn}
		\norm{F}_{\Hp^2}^2 = 4 \norm{\gamma_n}_{\ell^2(\N)}^2.
	\end{equation}
	We check that this choice of coefficients is optimal.
	Let the function $G(s) = \sum_{n \in \N} z_n n^{-s}$ be such that $\Re G(1/2 + \im t) = v(t)$ in the $L^2(I)$ sense. Write $z_n = x_n + \im y_n$. Then 
	\begin{equation*}
		v = \sum_{n \in \N} \left( \frac{x_n + \im y_n}{2}   n^{-1/2 - \im t} + \frac{x_n - \im y_n}{2} n^{-1/2 + \im t} \right)
	\end{equation*}
	gives an expansion for $v$ in the frame $\Omega^\ast$.
	Since the $\seq{\gamma_n}_{n \in \Z^\ast}$ are the coefficients of $v$ relative to the frame $\Omega^\ast_I$ given by the canonical
	dual frame and $\gamma_n = \overline{\gamma_{-n}}$, 
	 Lemma \ref{frame lemma II} implies that
	\begin{equation*} 
		   4\sum_{n \in \N} \abs{\gamma_n}^2 \leq \sum_{n \in \N} \frac{\abs{x_n + \im y_n}^2}{4} + \frac{\abs{x_{n} - \im y_{n}}^2}{4}.
	\end{equation*}
	But $\norm{F}^2_{\Hp^2}$ equals the left hand side, and the right-hand side is equal to  $ \sum_{n \in \N} \abs{z_n}^2 = \norm{G}^2_{\Hp^2}$. Hence, $\norm{F}_{\Hp^2} \leq \norm{G}_{\Hp^2}$. This establishes the optimality of the choice of coefficients.

	We turn to the explicit bounds of the lemma. Using the definition of the sequence $\seq{\gamma_n}$ and 
	Lemma \ref{frame lemma II}, we get
	\begin{equation*}
		\frac{1}{2\widetilde{B}_I} \norm{v}^2_{L^2(I)} \leq \sum_{n \in \N} \abs{\gamma_n}^2 \leq \frac{1}{2\widetilde{A}_I} \norm{v}^2_{L^2(I)} 
	\end{equation*}
	where $\widetilde{A}_I$ and $\widetilde{B}_I$ are the lower and upper frame bounds for $\Omega^\ast_I$.
	It is easily checked that 
	$A_I \leq  \widetilde{A}_I \leq 2A_I$ and $B_I \leq \widetilde{B}_I \leq 2B_I$ holds. By \eqref{eqn} this implies
	\begin{equation*}
		c_I \norm{v}_{L^2(I)}^2 \leq \norm{F}_{\Hp^2}^2 \leq  C_I \norm{v}^2_{L^2(I)},
	\end{equation*}
	with $B_I^{-1} \leq c_I \leq 2B_I^{-1}$ and $A_I^{-1}\leq C_I \leq 2A_I^{-1}$.
\end{proof}
We turn to the proof of Theorem \ref{frame theorem}.
What remains is to show the quantitative statements on the frame bounds. 
We denote the operator $\mathcal{S}_I$ by $\mathcal{S}_{2T}$. 
The inequalities (\ref{asymptotic lower bound}) are an immediate consequence of the two following lemmas.
\begin{lemma} \label{example}
	Let $\epsilon \in (0,1)$. For $T>2\pi\e$ we define the parameter $\mu >1$ through $T =(1 + \epsilon) \pi \mu^{-1} \e^{\mu}$ and let
	\begin{equation*}
	 G(x)  = \sin \pi x 
	 \prod_{k=2}^{[\e^{\mu}]} \frac{\sin \frac{\pi}{\log k} x}{\frac{\pi}{\log k}x}.
	\end{equation*}
	For $T$ large enough, 
	we have $g = \mathcal{F}^{-1}G \in L^2(-T,T)$ and
	\begin{equation*}
		\frac{ \norm{\mathcal{S}_{2T}^\ast g}_{\ell^2}^2}{ \norm{g}_{2}^2}
		\leq
		T^{- (1-\epsilon)\frac{2T}{\pi} \log \pi}. 
	\end{equation*}
\end{lemma}
\begin{proof}
	This is proved in section \ref{proof of lemma example}
\end{proof}
\begin{lemma} \label{construction}
	Let $\epsilon >0$ be given. For $T$ large enough and any given $f \in L^2(-T,T)$,  it is possible to choose an auxillary set of frequencies $\Lambda \subset \R$ in such a way that both
	inequalities
	\begin{equation} \label{easy inequality}
		 \sum_{n \in \N} \frac{\abs{\hat{f}(\log n)}^2}{n} 
		 \geq
		  \frac{1}{2T} \sum_{\lambda \in \Lambda} \abs{\hat{f}(\lambda)}^2,
	\end{equation}
	and
	\begin{equation} \label{hard inequality}
	 	\norm{f}^2_2 \leq  T^{(1+\epsilon)\frac{12T}{\pi}\log 2} \sum_{\lambda \in \Lambda} \abs{\hat{f}(\lambda)}^2
	\end{equation}
	are satisfied.
\end{lemma}
\begin{proof}
	This is proved in section \ref{proof of lemma construction}. Observe that the set $\Lambda$ in the statement may depend on $f$.
\end{proof}
To get the limits (\ref{asymptotics for small I}) we consider the  operator  $\mathcal{S}_I$. 
Recall that the upper and lower norm bounds for the operator $\mathcal{S}_I \mathcal{S}_I^\ast$ give exactly the upper and lower frame bounds for $\Omega_I$.
Next, we easily compute for $g \in L^2(I)$ that
\begin{equation*}
	 \mathcal{S} \mathcal{S}_I^\ast g  = R_I R_I^\ast g - \chi_I \int_I g(\tau) \dif \tau.
\end{equation*}
By the Cauchy-Schwarz inequality, the norm of the operator defined by the right-most term above goes to zero as $\abs{I} \rightarrow 0$. By the same argument, this holds for the right-most term in \eqref{harmonic embedding as convolution formula}, and we deduce that $\norm{\mathcal{S}_I \mathcal{S}_I^\ast - 2\pi \mathrm{Id}}$, where $\mathrm{Id}$ denotes the identity operator on $L^2(I)$, tends to zero as $\abs{I}$ tends to zero. The desired limits then follow.

We turn to the inequalities (\ref{upper bound inequalities}). 
Let $\seq{a_n}_{n \in \Z^\ast \backslash \{-1\}}$ be a sequence with only  finitely many non-zero coefficients. For some $N \in \N$ we have
\begin{equation} \label{onto: two-sided montgomery integral}
	 \norm{\mathcal{S}_I \seq{a_n} }^2 = \int_I \Abs{ \sum_{n=1}^N a_n n^{-1/2 - \im t} + \sum_{n=2}^N a_{-n} n^{-1/2 + \im t} }^2 \dif t.
\end{equation}
Finding the upper bound of this amounts to finding the upper frame bound.
This was essentially calculated by Montgomery using an inequality due to Montgomery and 
Vaughan \cite[eq. (27) p. 140]{montgomery94}. We state this inequality without proof 
and show how the bound follows from it.
\begin{lemma}[Montgomery and Vaughan] \label{montgomery-vaughan lemma}
	For $N \in \N$ let $\lambda_1, \dots, \lambda_N$ be distinct real numbers and $\delta_n = \min_{m \leq N, m\neq n} \abs{\lambda_m - \lambda_n}$. Then
	\begin{equation} \label{montgomery's inequality}
		 \sum_{n,m=1}^N \frac{x_n y_m}{\lambda_n - \lambda_m} 
		 \leq \gamma_0 \left( \sum_{n=1}^N \frac{\abs{x_n}^2}{\delta_n} \right)^{1/2} \left( \sum_{n=1}^N \frac{\abs{y_n}^2}{\delta_n} \right)^{1/2},
	\end{equation}
	where $\gamma_0 \leq 3.2$.
\end{lemma}
We follow Montgomery's argument. 
By applying the triangle inequality to \eqref{onto: two-sided montgomery integral}, it suffices to consider $G(s) = \sum_{n=1}^N a_n n^{-s}$ and $I = (-T,T)$.
Multiplying out $(\sum a_n n^{-s})(\overline{\sum a_n n^{-s}})$ and integrating, we get
	\begin{equation*}
		 \int_{-T}^T \Abs{ G\left(\frac{1}{2}+ \im t\right) }^2 \dif t
		 =
		 2T \sum_{n=1}^N  \frac{\abs{a_n}^2}{n} + 2 \sum_{n \neq m} \frac{a_n \overline{a}_m}{ \sqrt{nm}} \frac{\sin T \log \frac{n}{m}}{\log \frac{n}{m}}.
	\end{equation*}
	We apply inequality (\ref{montgomery's inequality}) to the second term on the right hand side 
	to find that this is less than or equal to
	\begin{equation*}
		 \left(2T   +  \frac{2 \gamma_0}{\log 2} \right) \sum_{n=1}^N \abs{a_n}^2.
	\end{equation*}
	Taking into account $\abs{I} = 2T$, we get the upper bound of (\ref{upper bound inequalities}). The lower bound follows by considering constant functions.

\section{Proof of Lemma \ref{example}} \label{proof of lemma example}

Our idea is to construct a function $g \in L^2(-T,T)$ for which the mass $\int_{-T}^T \abs{\hat{g}(\xi)}^2 \dif \xi$ is close to maximal, as centered as possible, and with  $\hat{g}(\pm \log n) = 0$ for $\pm \log n$ in a long interval containing $(-T,T)$.

Let $\epsilon > 0$ be fixed.
By the Paley-Wiener theorem, or observing that the inverse Fourier transform of $\sin (x)/x$ equals $(1/2)\chi_{[-1,1]}$ one checks that
\begin{equation*}
	\supp{ \mathcal{F}^{-1} \prod_{k=2}^{[\e^{\mu }]} \frac{\sin \frac{\pi}{\log k} x}{\frac{\pi}{\log k}x}} \subset \left(- \pi 
	\sum_{k=2}^{[\e^{\mu }]} \frac{1}{\log k}, \pi
	 \sum_{k=2}^{[\e^{\mu }]}  \frac{1}{\log k} \right).
\end{equation*}
We note that $\sum_{k=2}^{[\e^\mu]} \log^{-1} k / (\mu^{-1} \e^\mu) \rightarrow 1$ as $\mu \rightarrow \infty$. Hence, there is a number $\mu_0 >0$ such that for  $\mu > \mu_0$ we have
\begin{equation*}
	 \supp{\mathcal{F}^{-1}G} \subset   \Big(-(1+\epsilon) \pi \mu^{-1}\e^\mu,(1+\epsilon) \pi \mu^{-1}\e^\mu \Big).
\end{equation*}
Since the parameter $\mu$ is chosen by requiring that $T = (1 + \epsilon)\pi \mu^{-1}\e^\mu$, the function  $g = \mathcal{F}^{-1}G$ then satisfies $\supp{g} = (-T,T)$. 
The rest of the proof is split into two parts. First we find a lower bound for $\norm{g}_2$, then we compute an upper bound for $\norm{\mathcal{S}_{2T}^\ast g}_{\ell^2}$.

1) Our lower bound for $\norm{g}_2$ is crude and based on Bernstein's inequality. This inequality
says that for functions in the Bernstein space, i.e. functions $F \in L^\infty(\R)$ which satisfy $\mathcal{F}F \in L^\infty(-T,T)$, it holds that $\norm{F'}_\infty \leq T \norm{F}_\infty$. We let $F(x) = G(x)/\sin \pi x$. Clearly, $F$  is in the Bernstein space and satisfies $F(0) = 1$ and $\norm{F}_\infty \leq 1$. By the Bernstein inequality $\norm{F'}_\infty \leq T$, and so
\begin{equation*}
	 \abs{F(x)} \geq 1 - Tx,
\end{equation*}
for $x \in (-T^{-1}, T^{-1})$. Since $\sin \pi x \geq 2x$ for $x \in (0,1/2)$, and $T> 2$, we get   the estimate
\begin{equation} \label{useless estimate}
	 \norm{G}_{L^2(\R)}^2 \geq 8\int_{0}^{T^{-1}} x^2 (1 - Tx)^2 \dif x = \frac{4}{15 T^3}.
\end{equation}

2) We check the upper bound for $\norm{\mathcal{S}_{2T}^\ast g}_{\ell^2}$.
For $n \in \N$ such that $\log n > \mu$, we simply use $\abs{\sin x} \leq 1 $
and the inequality $\sum_{k=2}^{[\e^\mu]} \log \log k \leq \e^{\mu} \log \mu$  to get
\begin{align*}
	 \abs{G(\log n)} &\leq 
	 \prod_{k=2}^{[\e^\mu]} \frac{\log k}{\pi \log n} 
	\leq
	 \frac{\mu^{ \e^{\mu}}}{(\pi\log n)^{ \e^{\mu}-2}}.
\end{align*}
Summing over $\log n \geq \mu$,  we use the formula $\int_{\e^\mu}^\infty x^{-1} \log^{-\alpha} x\dif x  =  (\alpha-1)^{-1}  \mu^{1-\alpha}$ for $\alpha > 1$ to find for large $T$ that
\begin{align*}
	 \sum_{\log n \geq \mu} \frac{\abs{G(\log n)}^2}{n} 
	 &\leq 
	 \frac{\mu^{2\e^\mu}}{\pi^{2 \e^\mu - 4}} \sum_{\log n \geq \mu} \frac{1}{n \log n^{ 2\e^\mu - 4}} \\
	 &\leq
	 \exp\left\{ -2\log \pi\e^\mu\right\}.
\end{align*}
Since $G$ is an odd function satisfying $G(\log n) = 0$ for $n\geq 1$ with $\log n \leq \mu$, we may use the relation $T = (1+\epsilon)\pi \mu^{-1}\e^{\mu}$ and, again, play a game of epsilons to conclude that 
\begin{equation*}  
	 \norm{\mathcal{S}_{2T}^\ast g}_{\ell^2} \leq T^{ -(1-\epsilon)\frac{\log \pi}{\pi} T }.
\end{equation*}
The lemma now follows by combining this with the estimate (\ref{useless estimate}).

\section{Proof of Lemma \ref{construction}} \label{proof of lemma construction}

Let us start with a short outline of how we proceed. We fix $f \in L^2(-T,T)$ and  construct the corresponding  set $\Lambda$ 
of frequencies by perturbing the harmonic frequencies of the space $L^2(-W,W)$, where $W = (1 + \eta)T$ with $\eta \in (0,1)$. The perturbation is done so that the frequencies in $\Lambda $ coincide with  certain members of the set $\pm \log \N = \set{\log n \in \N : n \in \N} \cup \set{- \log n : n \in \N}$ for which the size of $\hat{f}$ is
minimal on certain intervals. 
In this way we ensure that (\ref{easy inequality}) holds. In the construction we choose an auxiliary index $k_0\sim c T\log T$
and perturb the harmonic frequencies with $|k| >k_0$ only
slightly, whereas the frequencies with $|k| \leq k_0$  are moved
more drastically.

The inequality (\ref{hard inequality}) is harder to establish. We combine an approach found in \cite{flornes_lyubarskii_seip99} with a well-known stability theorem due to M.I Kadec \cite{kadec64}, and some growth estimates due to S. A. Avdonin \cite{avdonin79}. The point is to ensure that the perturbation process used to construct $\Lambda$ is done in such a way that we get a representation of the type
\begin{equation} \label{onto: interpolation formula}
	\hat{f}(x) = \sum_{\lambda \in \Lambda} \hat{f}(\lambda) \Psi_\lambda (x),
\end{equation}
where the functions $\Psi_\lambda$ decay fast enough so that they essentially have disjoint supports. It is for this purpose that we need the over-sampling which is achieved by
considering $L^2(-T,T)$ as a subspace of $L^2(-W,W)$. 

\subsection{Constructing the set $\Lambda$}

Fix $f \in L^2(-T,T)$. We remark that although the set $\Lambda$ depends on $f$, the estimates below will be uniform over such sets.
Our starting point is the set of harmonic frequencies $\set{\pi k/W}_{k \in \Z}$ of the space $L^2(-W,W)$. The first fundamental fact is Kadec's $1/4$-Theorem which we state in the following lemma. (See e.g. \cite[p. 36, thm. 14]{young01} for a proof.)
\begin{lemma}[Kadec's $1/4$-Theorem]
	Let $\set{\mu_n}_{n \in \Z}$ be a sequence of real numbers such that $\abs{\mu_n - n} \leq \delta < 1/4$. Then $\set{\e^{\im \mu_n t}}_{n \in \Z}$ 
	forms a Riesz base for $L^2(-\pi,\pi)$ with bounds only depending on $\delta \in (0,1/4)$. 
\end{lemma}
In particular, by scaling, we find that if $\set{\mu_n}_{n \in \Z}$ is such that
$\abs{\mu_n - \pi n/W} \leq \rho < \pi/4W$ then $\set{\e^{\im \mu_n t}}_{n \in \Z}$ forms a Riesz base for $L^2(-W,W)$ satisfying
\begin{equation} \label{kadec estimate new}
	  W \norm{g}^2_{L^2(-W,W)} \simeq \sum_{n \in \Z} \abs{\inner{g}{\e^{\im \mu_n t}}}^2,
\end{equation}
where the implicit constants only depends on $W \rho\in (0,\pi/4)$. 

We split the construction of $\Lambda$ into two steps:

1) For each harmonic frequency $\pi k/W$, let
$I_k$ denote the open interval of radius $\rho = 1/4W < \pi/4W$ centered on this frequency, i.e.
\begin{equation*}
	 I_k = \left( \frac{\pi k - 1/4}{W},  \frac{\pi k + 1/4}{W} \right).
\end{equation*} 
For $n \in \N$, the distance between $\log n$ and $\log (n+1)$ is less than $n^{-1}$.
This means that for  $k\in \Z$ such that $\abs{\pi k/W} \geq \log  2W$,
the neighbourhoods $I_k$ will always contain a member of $\pm \log \N$. We define a threshold
\begin{equation*}
	 k_0 = \left[ \frac{W \log 2W}{\pi} \right].
\end{equation*}
Here the brackets denote the integer function. It follows that
for each  $k > k_0$ we may choose $n_k \in \N$ such that $\log n_k \in I_k$ and for which
\begin{equation*}
	 \abs{\hat{f}(\log n_k)} = \min_{n\in \N :\log n \in I_k } \abs{\hat{f}(\log n)}.
\end{equation*}
We do the corresponding selection for the negative frequencies. I.e., for $k < 0$ we choose $n_{k} \in \N$ such that $-\log n_{k} \in I_{-k}$ and for which 
\begin{equation*}
	 \abs{\hat{f}(-\log n_{k})} = \min_{n \in \N : \log n \in I_{k}} \abs{\hat{f}(-\log n)}.
\end{equation*}
Define
\begin{equation*}
	 \lambda_k = \left\{ \begin{array}{cl}  \log n_k  & \quad \text{if} \; k > k_0 \\
	 - \log n_{k} & \quad \text{if} \; k < -k_0
	 \end{array}\right.,
\end{equation*}
and let $\Lambda_w = \set{\lambda_k}_{\abs{k} > k_0}$. In particular
\begin{equation} \label{onto: definisjon av u}
	 U = \Set{\pi k/ W}_{\abs{k}\leq k_0} \cup \Lambda_w
\end{equation}
satisfies the Kadec theorem with $\rho W  = 1/4$. Hence, 
\begin{equation*}
	E_U = \set{\e^{\im \mu t}}_{\mu \in U}
\end{equation*}
forms a Riesz base for $L^2(-W,W)$ with absolute bounds.

2) Let $L_k$ denote the open interval of radius $1/4W$ centered on the point $\pi(k + k_0 + 1/2)/W$ for $k\geq 0$ and at $\pi(k - k_0 - 1/2)/W$ for $k < 0$. The intersection of  the sets $L_k$ with $\pm \log \N$ is non-empty. For $\abs{k} \leq k_0$, choose the number $n_k$ minimising the value of $\abs{\hat{f}( \mathrm{sgn}(n) \log \abs{n})}$ on $L_k \cap \pm \log \N$. Let
\begin{equation*}
	 \lambda_k = \left\{ \begin{array}{lrc} \log n_{k } & \quad \text{if} \; & \hspace{0.3 cm} 0 \leq k \leq k_0 \\ - \log n_{k } & \quad \text{if} \; & \hspace{-0.4 cm} - k_0 \leq k < 0 \end{array}  \right..
\end{equation*}
We denote $\Lambda_0 = \set{\lambda_k}_{\abs{k}\leq k_0}$ and set 
\begin{equation} \label{onto: definisjon av lambda}
	\Lambda = \Lambda_w \cup \Lambda_0. 
\end{equation}
Note that the intervals $I_k$ and $L_k$ are disjoint, and moreover that we have the separation
\begin{equation} \label{onto: separation}
	 \min_{\underset{n\neq m}{n \in \Z}} \abs{\lambda_m - \lambda_n } \geq \frac{1}{4W}.
\end{equation}
It now follows  from  \cite[Theorem 7, p. 129]{young01} that since $U$ is complete in $L^2(-W,W)$, then the same holds for  $\Lambda$.

\subsection{A sampling formula} \label{subsection: a sampling formula}
We will soon utilize the fact that the property of being a Riesz basis is stable under the arbitrary perturbation of a finite number of points as long as the new set of points is separated. 
This is a special case of a result by Avdonin in \cite{avdonin79}.
As we need explicit quantitative estimates, we are forced to
estimate the effect of perturbation in basic representation formulas.
For $f \in L^2(-T,T)$, a well-known sampling formula is the Whittaker-Kotel'nikov-Shannon formula
\begin{equation*}
	 \hat{f}(x) = \sum_{k \in \N} \hat{f}\left( \frac{\pi k}{T} \right) \frac{\sin T x}{(-1)^k T \left(x- \frac{\pi k}{T} \right)}.
\end{equation*}
This formula follows by taking the $L^2(\R)$ Fourier transform of both sides of the Fourier series 
\begin{equation*}
	f = \frac{\chi_{(-T,T)}}{2T} \sum_{k \in \N} \hat{f}(\pi k/T) \e^{\im \pi k t/T}.
\end{equation*}

There are two problems with the above  formula; we want to represent the function $f$ in terms of the frequencies $\Lambda$,
and it does not converge fast enough to separate the sampled coefficients from the rest of the terms in the way indicated above. 
Faced with a similar problem, it was realised by Flornes, Lyubarskii and Seip in \cite{flornes_lyubarskii_seip99} that the correct replacement for this formula is the Boas-Bernstein  formula (see e.g. \cite[p. 193]{boas54}). The formula says that if $\set{\lambda_k}_{k \in \Z}$ is a sequence of real numbers such that $\sup_{k \in \Z} \abs{\lambda_k - k} < \infty$, then for some $l \in \N$ large enough, it holds that
\begin{equation} \label{onto: boas-bernstein formula}
	 \hat{f}(x) = \sum_{k \in \Z} \hat{f}(\lambda) h_l(x - \lambda_k) \frac{G(x)}{G'(\lambda_k)(x- \lambda_k)},
\end{equation}
where
\begin{equation*}
	 G(x) = \prod_{\lambda \in \Z}  \; \hspace{-0.2 cm} '  \hspace{0.1 cm} \left( 1 - \frac{x}{\lambda_k} \right)
\end{equation*}
should be thought of as a sine type function while the factor
\begin{equation*}
	 h_l(x) = \left(\frac{\sin \eta x/l }{\eta x/l}\right)^l
\end{equation*}
helps the sum converge. The symbol $\prod'$ means that whenever $\lambda_k = 0$ the corresponding factor is taken to be $x$.  Clearly the hypothesis is satisfied for the sequence $\Lambda$. Our function $G$ is slightly more complicated than what was studied in \cite{flornes_lyubarskii_seip99}. 

We now explain for the readers convenience why this formula holds in our case, and at the same time we collect some explicit estimates that we are going to need.  Recall that, by Kadec's $1/4$-Theorem, the set of frequencies $U$ defined by \eqref{onto: definisjon av u} gives a Riesz base for 
$L^2(-W,W)$. Let
\begin{equation} \label{infinite product S}
	 S(z) := \prod_{\lambda \in U} \; \hspace{-0.2 cm} '  \hspace{0.1 cm} \left(1 - \frac{z}{\lambda}\right)=
	  z\prod_{k=1}^{k_0}\left( 1-\left(
	 \frac{z}{\pi k/w}\right)^2\right)\prod_{k>k_0}
	 \left( 1-\left(
	 \frac{z}{\lambda_k}\right)^2\right).
\end{equation}
The series clearly converges and
it is not hard to see that $S(z)$ is a function of exponential type. 
For instance, denote $\abs{z}=r$. Then for any $\epsilon > 0$ there exists constants $R>0$, $K>0$ and a polynomial $P(r)$, such that for $r > R$ we have
\begin{equation*}
	\abs{S(z)} \leq P(r) \prod_{k \geq K} \left( 1 + \frac{(1+\epsilon)^2 W^2 r^2}{\pi^2 k^2} \right) 
	\leq P(r) \sin \left( \im (1+\epsilon) W r\right) \lesssim \e^{(1+\epsilon)Wr}.
\end{equation*} 
This implies that $S(z)$ is at most of exponential type $W$.
By comparing it to the sine function with approximately the same zeroes, one is able to estimate the growth
along lines parallel to the real axis. This is the content of the following technical lemma, which is a slightly quantized special case of
\cite[Lemma 4]{avdonin79}.
\begin{lemma}[Avdonin] \label{avdonin lemma}
	Let $\Phi(z)$ be a function of the same type as (\ref{infinite product S}) with real zeroes $m + \delta_m$ satisfying $\sup_{n \in \N} \abs{\delta_m} < 1/2$. Then 
	there exist absolute constants such that
	\begin{equation*}
	 	\Abs{  \frac{\sin \pi (x + \im )}{\Phi(x + \im )}} 
		\simeq 
		\exp\left\{ \sum_{\abs{m} \leq 2\abs{x}, m \neq 0} \frac{\delta_m}{m} +  \frac{\delta_m}{x +\im h - m} \right\}, \quad x \in \R.
	\end{equation*}
\end{lemma}
In our case, we apply this lemma to $\Phi (x) = S( \pi x/W )$ with $\abs{\delta_k} \leq 1/4\pi$. A simple computation
now shows that 
\begin{equation*}
	\frac{1}{( 1+ \abs{x})^{1/\pi}} \lesssim \Abs{\Phi(x + \im )} \lesssim ( 1+ \abs{x})^{1/\pi}.
\end{equation*}
In particular, this implies that
\begin{equation} \label{onto: phi less than}
	 \Abs{\Phi(x) } \lesssim ( 1+ \abs{x})^{1/\pi}.
\end{equation}
Indeed,
\begin{equation*}
	 \Abs{ \frac{\Phi(x)}{\Phi(x + \im)} }^2  =  \prod_{m \in \N}\left( 1 - \frac{1}{1 +  (m + \delta_m - x)^2} \right).
\end{equation*}
Splitting this product into two parts, depending on whether $x \leq m$ or $x < m$, it is clear that it is bounded by some constant independent of $x$.
We may now infer from the definition of $\Phi(x)$ and \eqref{onto: phi less than} that
\begin{equation} \label{S upper inequality}
	 \Abs{S(x)} \lesssim  ( 1 + \abs{ Wx })^{1/\pi}.
\end{equation}

Since $ \pi^{-1} < 1/2$ it follows  for any $\lambda \in U$ that $S(x)/(S'(\lambda)(x-\lambda))$ is in $L^2(\R)$. Hence, by the Paley-Wiener theorem, this is the Fourier transform of 
a function $s_\lambda \in L^2(-W,W)$ that satisfies, by construction,
\begin{equation*}
	 \inner{s_\lambda}{\e^{\im \mu t}} = \left\{ \begin{aligned} 1 & \quad \text{if} \; \mu= \lambda \\ 0 & \quad \text{if} \; \mu \in U\setminus \{
	 \lambda \} . \end{aligned}\right.
\end{equation*}
Using this bi-orthogonality, and the fact that $E_U$ gives a Riesz basis, we immediately get the estimate
\begin{equation} \label{riesz basis inequality}
	 \int_\R \Abs{\frac{S(x)}{S'(\lambda)(x-\lambda)}}^2 \dif x \simeq W^{-1}.
\end{equation}
Moreover, we get the representation for $u \in L^2(-W,W)$,
\begin{equation} \label{S riesz basis formula}
	 \hat{u}(x)  = \sum_{\lambda \in U} \hat{u}(\lambda) \frac{S(x)}{S'(\lambda) (x- \lambda)}.
\end{equation}
Finally, we  set
\begin{equation} \label{formula for G}
	 G(z) = S(z) \prod_{1\leq \abs{k} < k_0}\left( \frac{1 - \frac{z}{\lambda_k}}{1 - \frac{z}{\pi k/W}} \right) = \prod_{\lambda \in \Lambda}  \; \hspace{-0.2 cm} '  \hspace{0.1 cm} 
	 \left( 1 - \frac{z}{\lambda}\right).
\end{equation}
As we have replaced
just a finite number of elements in the basis $E_U$ by new vectors, and
the new set $E_\Lambda$ is still complete, it is still a basis. Moreover, since
$E_U$ is a Riesz basis, then $E_\Lambda$ is also one. From the fact that the functions $G(x)/(G'(\lambda)(x-\lambda))$
are biorthogonal to the  basis by definition (and belong to $L^2(-W,W)$ by
construction), the uniqueness of the biorthogonal
system immediately gives the formula
\begin{equation} \label{onto: pre G formula}
	 \hat{u}(x)  = \sum_{\lambda \in \Lambda} \hat{u}(\lambda) \frac{G(x)}{G'(\lambda) (x- \lambda)}.
\end{equation}
We now use the oversampling. Since our $f$ belongs to $ \in L^2(-T,T)$, it follows that the function
$y\mapsto\hat{f}(x) h_l(y-x)$, where
\begin{equation*}
	h_l(x): = \left(\frac{\sin \eta x/l}{\eta x/l} \right)^l,
\end{equation*}
is the Fourier transform of a function in $L^2(-W,W)$, by the relation
between $W,\eta$ and $T\geq 1.$
Hence, we may apply formula \eqref{onto: pre G formula} to this function, and by substituting $y=x$, the Boas-Bernstein formula \eqref{onto: boas-bernstein formula} follows.
\subsection{The inequality (\ref{easy inequality})}

As before, $f \in L^2(-T,T)$ and $\Lambda$ is the set of frequencies constructed in the previous section.
It is clear that
\begin{align*}
	 \sum_{n\in \N} \frac{\abs{\hat{f}(\log n)}^2}{n} 
	 &\geq 
 	 \sum_{\log n \in \cup L_k} \frac{\abs{\hat{f}(\log n)}^2}{n}
		+
	 \sum_{\log n \in \cup I_k} \frac{\abs{\hat{f}(\log n)}^2}{n}.
\end{align*}
By the choice of the frequencies $\Lambda_w$,
followed by elementary estimates, this is seen to be bigger than
\begin{equation*}
	\left( \sum_{\abs{k} \leq k_0} \abs{\hat{f}(\lambda_k)}^2
	\sum_{\log n \in L_k} \frac{1}{n} \right)
		+
	\left( \sum_{\abs{k}>k_0} \abs{\hat{f}(\lambda_k)}^2 
	 \sum_{\log n \in I_k} \frac{1}{n} \right)
	 \geq
	 \frac{1}{4T} \sum_{\lambda \in \Lambda} \abs{\hat{f}(\lambda)}^2,
\end{equation*}
as was to be shown.

\subsection{The inequality (\ref{hard inequality}) }

Recall the  relation $W = (1+\eta)T$ for $\eta \in (0,1)$. 
We now specify $\eta:=\varepsilon $ .
Up to now this only effects the estimates  via the
auxiliary function $h_l$. We choose $l=2$ and write $h=h_2$.
With this choice let
\begin{align*}
	F_0 &= \sum_{\lambda \in \Lambda_0} \hat{f}(\lambda) h(x-\lambda) \frac{G(x)}{G'(\lambda) (x-\lambda)}
	\quad \text{and}
	 \\
	F_w &=\sum_{\lambda \in \Lambda_w} \hat{f}(\lambda) h(x-\lambda) \frac{G(x)}{G'(\lambda) (x-\lambda)}.
\end{align*}
By the Boas-Bernstein formula \eqref{onto: boas-bernstein formula}, we have $\hat f = F_0 + F_w$, from which it
follows that $\norm{f}_2 =\sqrt{2\pi}\norm{\hat f}\leq \norm{F_0}_2 + \norm{F_w}_2$.
The inequality \eqref{hard inequality} follows from the estimates
\begin{equation*}
	\norm{F_0}^2 \lesssim  T^{(1+\epsilon) \frac{12 T}{\pi}\log 2 }  \sum_{\lambda \in \Lambda_0} \abs{\hat{f}(\lambda)}^2,
\end{equation*}
and
\begin{equation} \label{Fw estimate}
	\norm{F_w}^2 \lesssim T^{(1+\epsilon) \frac{12 T}{\pi}\log 2 } \sum_{\lambda \in \Lambda_w} \abs{\hat{f}(\lambda)}^2.
\end{equation}
These estimates are valid for $T>1$ and the implicit constants depend just on $\epsilon$.
The proofs are essentially the same, so we only explain how to get (\ref{Fw estimate}). We collect some technical estimates
in the following two lemmas.
\begin{lemma}  \label{S estimates} 
	For $\lambda \in \Lambda_w \cup \set{\frac{\pi k}{W}}_{\abs{k_0} \leq k}$, we have 
	\begin{align} 
	\label{avdonin inequality}
	 W^{1-1/\pi}(|\lambda |+1)^{-1/\pi}\lesssim \Abs{S'(\lambda)} &\lesssim  W^{1+1/\pi}(|\lambda |+1)^{1/\pi}.
	\end{align}
\end{lemma}
\begin{proof}
	The estimate follows from Lemma \ref{avdonin lemma}
	in a similar way as the inequality (\ref{S upper inequality}). 
	As before, we set
	$\Phi(x) = S( \pi x /W)$. We write the zeroes of $\Phi$ in the form $\mu_m = m + \delta_m$, whence
$\abs{\delta_m} \leq 1/4\pi$.
	In addition, set $\Phi_m(x) = \Phi(x)/(x-\mu_m)$.
	We then have $\Phi'(\mu_m) = \Phi_m(\mu_m),$ and $\Phi_m(\mu_m + \im ) = -\im  \Phi(\mu_m + \im y)$.
	Combining these two formulas we find
	\begin{equation*}
	 	\Phi'(\lambda) = \im \frac{\Phi_m(\mu_m)}{\Phi_m(\mu_m + \im )} \Phi(\mu_m + \im ).
	\end{equation*}
	Since Lemma \ref{avdonin lemma} 
	implies that $(1 + \abs{x})^{-1/\pi} \lesssim \abs{\Phi(x + \im )} \lesssim (1 + \abs{x})^{1/\pi}$, 
	the formula (\ref{avdonin inequality}) follows as soon as we know that
	\begin{equation*}
		 \Abs{\frac{\Phi_m(\mu_m)}{\Phi_m(\mu_m + \im)}} \simeq 1,
	\end{equation*}
	with the implicit constants independent of $m \in \N$.	 But this follows by simply expanding the left hand side
	into an infinite product, and using the properties of the $\mu_k$:
	\begin{equation*}
		 \prod_{k \neq m} \; \hspace{-0.2 cm}'\hspace{0.1 cm} \Abs{\frac{1 - \frac{\mu_m}{\mu_k}}{1 - \frac{\mu_m + \im }{\mu_k} }} = \left(  \prod_{k \neq m}  1 + \frac{1}{(\mu_k - \mu_m)^2}   \right)^{-1/2}\simeq 1.
	\end{equation*}
\end{proof}
We establish some notation.
Let $J_m = \left[ (m-\frac{1}{2}) \frac{\pi}{W}, (m+\frac{1}{2})\frac{\pi}{W} \right)$.
This means that 
\begin{equation*}
	J = \bigcup_{\abs{m} \leq k_0} {J_m} = \left[\lambda_{-k_0} - \frac{\pi}{2W}, \lambda_{k_0} + \frac{\pi}{2W}\right).
\end{equation*}
Since $J_m \cap J_n = \emptyset$, this is a partition of the interval $J$.
\begin{lemma} \label{lemma full of estimates}
	Let $\epsilon >0$. For any $\lambda_k \in \Lambda_w$ and $W>1$,
	\begin{align}
		\label{fourth estimate}
		&\Abs{\prod_{{\lambda_n \in \Lambda_0}} \frac{\lambda_k - \frac{\pi n}{W}}{\lambda_k - \lambda_n}}
		\hspace{-2 cm}
		&\lesssim&  \quad 
		W^{(1+\epsilon)\frac{2W}{\pi} \log 2}. \\
	\end{align}
	Moreover,
	\begin{align}
	\label{fifth estimate}
		&\Norm{\prod_{\lambda_n \in \Lambda_0} \frac{x - \lambda_n}{x-\frac{\pi n}{W}} }_{L^\infty(\R \backslash J)}
		\hspace{-2 cm}
		&\lesssim& \quad 
		W^{(1+\epsilon)\frac{W}{\pi}\log \frac{3^3}{2^4}} 
	\end{align}
		and for $m \in \set{-k_0, \ldots, k_0}$,
		\begin{align}
		\label{sixth estimate}
		&\Bigg\|\prod_{\underset{n \neq m}{\lambda_n \in \Lambda_0}}\frac{x-\lambda_n}{x-\frac{n \pi}{W}}\Bigg\|_{L^\infty( J_m)}
		\hspace{-2 cm}
		&\lesssim& \quad 
		W^{(1+\epsilon)\frac{4W}{\pi}\log 2}.
	\end{align}
	We note that here the constants depend only on $\epsilon$.
\end{lemma}
\begin{proof}
	The arguments for all the inequalities are basically the same, so we only give the one for (\ref{sixth estimate}).
	We first consider the case $m=0$.
	\begin{multline*}
		\Bigg\|\prod_{\underset{n \neq m}{\lambda_n \in \Lambda_0}}\frac{x-\lambda_n}{x-\frac{n \pi}{W}}\Bigg\|_{L^\infty( J_m)}
		\leq
		\prod_{\underset{n>0}{\lambda_n \in \Lambda_0}} \frac{\lambda_n-\frac{\pi}{2W}  }{  \frac{\pi}{W}n-\frac{\pi}{2W} }
		\prod_{\underset{n<0}{\lambda_n \in \Lambda_0}}\frac{|\lambda_n|-\frac{\pi}{2W}  }{  \frac{\pi}{W}|n|-\frac{\pi}{2W} }
		\\
		\leq
		\prod_{n=1}^{k_0} \left(\frac{ n + k_0 + 1/2}{n - \frac{1}{2}} \right)^2
		\leq
		4 \left\{ \frac{(2k_0+1)!}{(k_0+1)! (k_0-1)!} \right\}^2
	\end{multline*}
	By Stirling's formula this is smaller than some constant times $\e^{(1+\epsilon) 4k_0\log 2 }$.
	Since $k_0 = \left[ W \log 2W/\pi\right]$, we obtain the desired inequality.
	Elementary considerations imply that 
	the largest value basically corresponds to the case $m=0$. Hence, the inequality follows for $m \in \set{-k_0, \ldots, k_0}$.
\end{proof}

We now resume the proof of the inequality (\ref{hard inequality}). 
The first thing to notice is that the factor $h$ allows us to use the Cauchy-Schwarz inequality.
This is essentially why we constructed the set $\Lambda$ to be over-sampling for $L^2(-T,T)$. 
We get
\begin{multline} \label{second F1 estimate}
	 \norm{F_w}^2_{L^2(\R)}
	 =
	 \int_\R \Abs{ \sum_{\lambda \in \Lambda_w} \hat{f}(\lambda) h(x-\lambda) \frac{G(x)}{G'(\lambda) (x-\lambda)}  }^2 \dif x \\
	 \leq
	 	 \int_\R  \sum_{\lambda \in \Lambda_w} \abs{ \hat{f}(\lambda)}^2 \abs{h(x-\lambda)} \Abs{ \frac{G(x)}{G'(\lambda) (x-\lambda)} }^2 \times \sum_{\mu \in \Lambda_w}  \abs{h(x-\mu)} \dif x \\
		 \leq
	 	 \Norm{\sum_{\mu \in \Lambda_w}  h(x-\mu) }_{L^\infty(\R)}   \sum_{\lambda \in \Lambda_w} \abs{ \hat{f}(\lambda)}^2 
		 \underbrace{\int_\R \abs{h(x-\lambda)} \Abs{ \frac{G(x)}{G'(\lambda) (x-\lambda)} }^2 \dif x}_{(I)}.
\end{multline}
It is readily checked that the factor outside of the sum is less than some absolute constant, so it only remains to deal with (I). For $\lambda_k \in \Lambda_w$, we get
by expanding $G(x)$ in terms of $S(x)$, and using the inequality (\ref{fourth estimate}), that
\begin{multline} 
	(I)
	= \int_\R  \Abs{ \frac{S(x)}{S'(\lambda_k) (x-\lambda_k)} }^2 
	\abs{h(x-\lambda_k)}
	\Abs{ \prod_{\lambda_n \in \Lambda_0} \frac{x - \lambda_n}{x - \frac{\pi n}{W}} 
	\prod_{\lambda_n \in \Lambda_0} \frac{\lambda_k-  \frac{\pi n}{W} }{\lambda_k - \lambda_n} 
	}^2
	\dif x 
	\\
	\lesssim
	W^{(1+\epsilon)\frac{4W}{\pi} \log 2 }
	\underbrace{\int_\R  \Abs{\frac{S(x)}{S'(\lambda_k) (x-\lambda_k)}}^2 \abs{h(x-\lambda_k)}
	\prod_{\lambda_n \in \Lambda_0} \Abs{\frac{x- \lambda_n}{ x - \frac{\pi n}{W} }}^2 \dif x}_{(II)}.
\end{multline}
Here and below, the inequalities are valid for $W>1$ and the implicit constants depend on $\epsilon$.
Recall that $J_m = \Big[(m-1/2)\frac{\pi}{W}, (m+1/2) \frac{\pi}{W} \Big)$ and $J = \cup_{m=-k_0}^{k_0} J_m$. Then 
\begin{multline*}
	 (II)  = 
	 \underbrace{\int_{\R \backslash \cup J_m} \Abs{\frac{S(x)}{S'(\lambda_k) (x-\lambda_k)}}^2 \abs{h(x-\lambda_k)} \prod_{\lambda_n \in \Lambda_0} \Abs{\frac{x - \lambda_n}{ x- \frac{\pi n}{W} }}^2 \dif x}_{(III)}
	 \\
	 +
	 \sum_{\abs{m} \leq k_0} \Abs{ \frac{S'(\frac{\pi m}{W} ) }{S'(\lambda_k)} }^2 \underbrace{ \hspace{0 cm}\int_{J_m} \Abs{\frac{S(x)}{S'(\frac{\pi m}{W} )(x- \frac{\pi m}{W} )}}^2 
	 \times \abs{h(x-\lambda_k)} \Bigg|\frac{x - \lambda_m}{x-\lambda_k} \prod_{\underset{n\neq m}{\lambda_n \in \Lambda_0}} \frac{x - \lambda_n}{x-\frac{\pi n}{W} } \Bigg|^2  \dif x}_{(IV)}.
\end{multline*}
By the bound $\abs{h(x)} \leq 1$, and the inequalities    \eqref{riesz basis inequality} and \eqref{fifth estimate}, this implies that $(III) \lesssim W^{(1+\epsilon)\frac{2W}{\pi}\log \frac{3^3}{2^4}}$. 
It is readily checked that for $T\geq 1$ and $\abs{k} > k_0$ one has $\norm{h(x-\lambda_k)}_{L^\infty(J)} \lesssim (\abs{k} - k_0)^{-2}$ for
some constant that depends only on $\epsilon$. Using this, in addition to the estimates \eqref{riesz basis inequality},  \eqref{avdonin inequality} and \eqref{sixth estimate}, we get
\begin{eqnarray*}
	 (IV)& \leq&  \frac{\norm{h(x-\lambda_k)}_{L^\infty(\cup J_m)}}{\abs{S'(\lambda_k)}^2} 
	  	 \sum_{\abs{m} \leq k_0} \left(\abs{S'(\frac{\pi m}{W})}^2 \Norm{\frac{x - \lambda_n}{x-\lambda_k}  }_{L^\infty(J_m)}\right.
	 \\
	 &&\times \;
		\Bigg\| \prod_{\underset{\lambda_n \in \Lambda_0}{ n\neq m}} \frac{x - \lambda_n}{x-\frac{\pi n}{W}} \Bigg\|^2_{L^\infty(J_m)}
	\left.   \int_\R \Abs{\frac{S(x)}{S'(\frac{\pi m}{W}) (x-\frac{\pi m}{W})}}^2 \dif x\right)
	 \lesssim
	 W^{(1+\epsilon) \frac{8W}{\pi} \log 2}.
\end{eqnarray*}
By combining the inequalities for (I)-(IV) with \eqref{second F1 estimate}, we get for $W>1$ the estimate
\begin{equation*}
	 	\norm{F_w}^2 \lesssim W^{(1+\epsilon )\frac{12W}{\pi} \log 2} \sum_{\lambda_k \in \Lambda} \abs{\hat{f}(\lambda_k)}^2.
\end{equation*}
The implicit constant depends on $\epsilon$. By using the fact that $\varepsilon >0$ is arbitrary, the relation $W = (1+{\epsilon})T$, and the simple inequality $(x(1+\varepsilon ))^{x(1+\varepsilon )}\leq x^{x(1+2\varepsilon )}$, valid for large enough $x$ and small enough
$\epsilon >0$, the  inequality \eqref{Fw estimate} follows.

\section{Proof of Theorem \ref{sobolev matching real parts}} \label{section: mccarthy spaces}

In order to proceed, we define the  Sobolev space $W^\alpha(I)$ on an interval $I$.  For $\alpha \in \R$, let $w_\alpha(\xi) = (1 + \abs{\xi}^2)^{\alpha/2}$ and denote the space of tempered distributions by $\mathcal{S}'(\R)$. The \emph{unrestricted} Sobolev space is given by
\begin{equation*}
	 W^\alpha(\R) = \Set{ u \in \mathcal{S}'(\R) : \; \hat u\in L^2_{\mathrm{loc}}\; \mbox{and}\; \| u\|_{W^\alpha}:=\int_\R \abs{\hat{u}(\xi)}^2 w_{2\alpha}(\xi) \dif \xi < \infty}.
\end{equation*}
For an open and (possibly unbounded) interval $I \subset \R$ , we let $W^\alpha_0(I)$ be the subspace of $W^\alpha(\R)$ that consists of distributions having support in $I$. By a scaling and mollifying argument one easily checks that this subspace coincides with the closure of $\mathcal{C}^\infty_0(I)$ in the norm of $W^\alpha(\R)$. The definition of $W^\alpha_0(I)$ remains unchanged if $I$ is a finite union of separated intervals.

With this, we define the  Sobolev space  
\begin{equation*}
	 W^\alpha(I) := W^\alpha(\R)/ W^{\alpha}_0(\R \backslash \bar{I}^C).
\end{equation*}
In other words, the quotient space $W^\alpha(I)$  contains the restrictions of distributions in $W^\alpha(\R)$ to the interval $I$ with the norm
\begin{equation*}
	\norm{u}_{W^\alpha(I)}  =  \inf_{\underset{v|_I = u}{v \in W^\alpha(\R)}} \norm{v}_{W^\alpha(\R)}.
\end{equation*}
Under the natural pairing $(u,v) = \int_\R \hat{u}(\xi) \hat{v}(\xi) \dif \xi$, 
the dual space of $W^\alpha(I)$ is isometric to $W^{-\alpha}_0(I)$, as is readily verified.
It is well-known that the functions  in the spaces $D_{\alpha}(\C_{1/2})$ have distributional boundary values that belong to the Sobolev spaces $W^{\alpha/2}(I)$ on bounded and open intervals $I \subset \R$. By the local embeddings, the same holds true for the spaces
\begin{equation*}
	 \Hp^2_\alpha = \Set{\sum_{n \in \N} a_n n^{-s} : \sum_{n \in \N} \abs{a_n}^2 w_{\alpha} (\log n) < \infty}.
\end{equation*}
Closely related to the space $\Hp^2_\alpha$ is the sequence  space $\ell^2_\alpha(\Z^\ast)$. We define it to be the sequences of complex numbers $(a_n)$ finite in the norm
\begin{equation*}
	 \norm{(a_n)}_{\ell^2_\alpha} = \Big( \sum_{n \in \N}  {\abs{a_n}^2 w_{\alpha}(\log n) + \abs{a_{-n}}^2 w_{\alpha}(-\log n)} \Big)^{\frac{1}{2}}.
\end{equation*}
The following  is analogue to Lemma \ref{RI is onto}. Recall that
\begin{equation*}
	 R_I : (a_n)_{n \in \Z^\ast} \longmapsto  \left( \sum_{n \in \N} \frac{a_n n^{-\im t} + a_{-n}n^{\im t}}{\sqrt{n} } \right) \Bigg|_I.
\end{equation*}
\begin{lemma} \label{onto: sobolev RI is onto}
	The operator $R_I : \ell^2_\alpha \longrightarrow W^{\alpha/2}(I)$,
	originally defined only on finite sequences, extends to a bounded and onto operator.
\end{lemma}
Let $R_I^\ast$ denote the adjoint operator of $R_I$ with respect to the natural pairing of the Sobolev spaces and  of  $\ell^2_\alpha$.
We remark that Lemma \ref{onto: sobolev RI is onto} says that
$R_I : \ell^2_\alpha \longrightarrow W^{\alpha/2}(I)$
is both bounded and surjective.
This is equivalent to saying that the operator $R_I^\ast : W^{-\alpha/2}_0(I) \rightarrow \ell^2_{-\alpha}$ is bounded and bounded below in norm. Since
\begin{equation*}
	 \norm{R_I^\ast g}_{\ell^2_{-\alpha}}^2 = 	 	\sum_{n \in \N} \frac{\abs{\hat{g}(\log n)}^2w_{-\alpha}(\log n) + \abs{\hat{g}(-\log n)}^2w_{-\alpha}(\log n)}{n},
\end{equation*}
we note that Lemma \ref{onto: sobolev RI is onto} may be formulated as follows (observe that we have changed $\alpha$ to $-\alpha$
for notational convenience).
\begin{lemma} \label{sobolev frame lemma}
	Let $I \subset \R$ be a bounded interval and $\alpha\in\R.$
	Then there exist constants, depending only on the length of $I$, such that for any $f\in W_0^{\alpha /2}$ one has
	\begin{equation} \label{sobolev frame inequalities}
	 	\sum \frac{\abs{\hat{f}(\log n)}^2w_\alpha(\log n) + \abs{\hat{f}(-\log n)}^2w_\alpha(\log n)}{n}  \simeq
		\norm{f}^2_{W^{\alpha/2}_0(I)}.
	\end{equation}
\end{lemma}

In place of Lemma \ref{second stability theorem} from Semi-Fredholm theory,  we will use the following more crude result, which we prove for the readers convenience
	\begin{lemma}\label{onto: subspace lemma}
   		Let $H,K$ be Hilbert spaces and $A : H \rightarrow K$ be a
    	bounded and injective operator. If there exist a subspace
    	$M \subset H$ of finite co-dimension, and a constant $C>0$
		such that $\norm{Af} \geq C \norm{f}$ for all $f \in M$,
    	then $A$ is bounded below in norm on $H$.
	\end{lemma}
	\begin{proof}[Proof of lemma \ref{onto: subspace lemma}]
	Assume that there exists a sequence of vectors $\seq{f_n}$ such that $\norm{f_n} \equiv 1$ and $\norm{Af_n} \leq n^{-1}$. We obtain a contradiction by showing that the sequence $\seq{f_n}$ converges in norm.

	Let $P_M$ and $P_{M^\perp}$ be the orthogonal projections onto $M$ and $M^\perp$, respectively.
	Since $M^\perp$ is of finite dimension, we assume that $P_{M^\perp} f_n$ converges to some vector $g \in H$. 
	Moreover, by the triangle inequality,
	\begin{equation*}
		\norm{AP_M f_n + Ag} \leq \frac{1}{n} + \norm{A(P_{M^\perp}f_n - g)}.
	\end{equation*}
 	It follows that $AP_M f_n$ converges to $-Ag$, and in particular $\seq{AP_M f_n}$ has to be a Cauchy sequence.
	The final step is to observe that the lower boundedness of $A$ on $M$ implies
	\begin{equation*}
		\norm{AP_M (f_n - f_m)} \geq C \norm{P_Mf_n - P_Mf_m} 
	\end{equation*}
	for some $C>0$.
	I.e., $\seq{P_Mf_n}$ is a Cauchy sequence. And so, since $f_n = P_M f_n + P_{M^\perp} f_n$, we get that $\seq{f_n}$ is a Cauchy sequence, as was to be shown.	
	\end{proof}
	We are now ready to prove the lemma.
	\begin{proof}[Proof of Lemma \ref{sobolev frame lemma}]
	It is enough to prove the relation 
	\eqref{sobolev frame inequalities}.
	For $f \in C^\infty_0(I)$, it is clear that this expression converges. 
	Since $n^{-1} \log 2 \leq \log((n+1)/n)$ and $w_\alpha(\log n) \leq w_\alpha(\xi)$ for $\xi \in (\log n, \log n+1)$, it follows that the left hand side of \eqref{sobolev frame inequalities} is less than
	the constant $1/\log 2$ times
	\begin{equation*}
		\sum_{n \in \N} \left( \int_{\log n}^{\log n+1}\abs{\hat{f}(\log n)}^2w_\alpha(\xi) \dif \xi + \int_{-\log n+1}^{-\log n}\abs{\hat{f}(-\log n)}^2w_\alpha(\xi) \dif \xi \right).
	\end{equation*}
	Adding and substracting by $\hat{f}(\xi) w_{\alpha}(\e^\xi)$ within the absolute value signs, and
	using the inequality $\abs{x+y}^2 \leq 2(\abs{x}^2 +  \abs{y}^2)$, shows that this again is smaller 
	than the constant $2$ times
	\begin{multline} \label{onto: inequality: weighted embedding}
		\int_\R \abs{\hat{f}(\xi)}^2 w_\alpha(\xi) \dif \xi 
		+
		\sum_{n \in \N} \left( \int_{\log n}^{\log n+1}\abs{\hat{f}(\xi) - \hat{f}(\log n)}^2w_\alpha(\xi) \dif \xi \right. \\ + \left. \int_{-\log n+1}^{-\log n}\abs{\hat{f}(\xi) - \hat{f}(-\log n)}^2 w_\alpha(\xi) \dif \xi \right).
	\end{multline}
	The first term is simply
		$\norm{f}_{W^{\alpha/2}_0}^2$.
	We need to show that the second term is also controlled by this norm. 
	By expanding the Fourier transform, we find that
	\begin{equation}
	\label{onto: not so useless estimate II}
		\abs{\hat{f}(\log n) - \hat{f}(\xi)}^2
		=
		\Big|\int_{\log n}^\xi \hat{f}'(\tau) \dif \tau\Big|^2
		\lesssim  \frac{1}{n \; w_\alpha(\log n)} \int_{\log n}^{\log n +1} \abs{\hat{f}'(\tau)}^2 w_\alpha(\tau) \dif \tau,
	\end{equation}
	Inserting this into the last term of \eqref{onto: inequality: weighted embedding}
	yield the upper bounds 
	\begin{multline*}
		\sum_{n \in \N} \frac{1}{n} \left(\int_{\log n}^{\log n+1} \abs{\hat{f}'(\tau)}^2 w_\alpha(\tau) \dif \tau
		 + \int_{-\log n}^{-\log n+1} \abs{\hat{f}'(\tau)}^2 w_\alpha(\tau) \dif \tau\right)
		 \\
		 \leq
		 \int_\R \abs{\hat{f}'(\tau)}^2 w_\alpha(\tau) \dif \tau
		 =  \norm{t f}^2_{W_0^{\alpha/2}(I)}.
	\end{multline*}
    The last equality follows by the rule for differentiating a Fourier transform 
    and the definitions of the norms.  Since multiplication by $t$ is continuous
    on $W^{\alpha}_0(I)$ (choose $\phi\in C_0^\infty (\R)$ with $\phi(t)=t$ in a neighbourhood of $I$, and observe that multiplication by
    $\phi$ in continuous on $W^{\alpha/2}(\R )$) it follows that $\norm{tf}_{W^\alpha_0(I)} \lesssim \norm{f}_{W^{\alpha/2}_0(I)}$. 
    This proves the upper inequality.

   We turn to the lower inequality. We need to be slightly more careful. 
    By the same argument as above, we find that the left hand side of \eqref{sobolev frame inequalities} is greater than
    \begin{equation*}
		\norm{f}_{W^{\alpha/2}_0(I)}^2 - C_\alpha \norm{tf}_{W^{\alpha/2}_0(I)}^2.
	\end{equation*}
	However, for general $\alpha$ and $I$ this leaves us with something negative.
	The solution is 
	for some sufficiently large $N \in \N$ to leave the terms with $\abs{n} < N$ out of the sum 
	on the left hand side of \eqref{sobolev frame inequalities}. This only makes the sum smaller,
	and applying again inequality (\ref{onto: not so useless estimate II}) as before, \eqref{sobolev frame inequalities} is seen to be greater than
	\begin{eqnarray*}
		&&\norm{f}^2_{W^\alpha_0(I)} - \int_{-\log N}^{\log N}  \abs{\hat{f}(\xi)}^2 w_\alpha(\xi) \dif \xi
		\\&& -
		        C_\alpha \sum_{\abs{n} \geq N} \frac{1}{n} \left( \int_{\log n}^{\log n+1}  \abs{\hat{f}'(\tau)}^2 w_\alpha(\tau)^2 \dif \tau \right. + \left. \int_{-\log n+1}^{-\log n} \abs{\hat{f}'(\tau)}^2 w_\alpha(\tau)^2 \dif \tau \right).
	\end{eqnarray*}
	By the continuity of 
	multiplication by the independent variable,
	given any $\epsilon >0$, we may choose $N$ large enough so that this is greater than
	\begin{equation*}
		(1- \epsilon) \norm{f}^2_{W^\alpha_0(I)} - \int_{-\log N}^{\log N}  \abs{\hat{f}(\xi)}^2 w_\alpha(\xi) \dif \xi.
	\end{equation*}

	Next, we explain how to use Lemma \ref{onto: subspace lemma} to conclude.
	The lemma says that it is sufficient to find a subspace $\mathcal{M} \subset W^{\alpha/2}_0(I)$ with finite co-dimension such that for all $f \in \mathcal{M}$ we have
	\begin{equation} \label{onto: blabla}
		 \int_{-\log N}^{\log N}  \abs{\hat{f}(\xi)}^2 w_\alpha(\xi) \dif \xi \leq \frac{1}{2} \norm{f}^2_{W^{\alpha/2}_0(I)}.
	\end{equation}
	For $\eta > 0$ and $K > {1}/{2\eta}$, choose a finite sequence of strictly increasing real numbers $(\xi_n)_{n=1}^{K+1}$ such that $\xi_1 =   -\log N$, $\xi_{K+1} = \log N$ and $\inf_{n \in J} \abs{\xi - \xi_n} < \eta$. We set
	\begin{equation*}
		\mathcal{M}  =  \Set{f \in W^\alpha_0(I)  :  \hat{f}(\xi_n) = 0, \quad \text{for} \; 1 \leq n \leq K+1}. 
	\end{equation*}
	This is a subspace of finite co-dimension in $W^{\alpha/2}_0(I)$. Moreover, by choosing $\eta$ small enough, an estimate of the type \eqref{onto: not so useless estimate II} now implies \eqref{onto: blabla}. Indeed, for $f \in \mathcal{M}$, the left-hand side is equal to
	\begin{equation*}
		 \sum_{n = 1}^K \int_{\xi_n}^{\xi_{n+1}} \abs{\hat{f}(\xi) - \hat{f}(\xi_n)}^2w_\alpha (\xi )\dif \xi
		 \lesssim
		 \eta^2 \sum_{n = 1}^K  \int_{\xi_n}^{\xi_{n+1}} \abs{\hat{f}'(\tau)}^2 w_\alpha(\tau) \dif \tau
		 \leq 
		 \eta^2 \norm{t f}^2_{W^{\alpha/2}_0(I)}.
	\end{equation*}
	By the continuity of multiplication by the independent variable, the assertion now follows.
	
\end{proof}
\begin{proof}[Proof of Theorem \ref{sobolev matching real parts}]
	Let $f(1/2 + \im t$) denote the boundary distribution of $f \in D_\alpha(\C_{1/2})$. 
	Moreover, let $v$ be the real part of $f(1/2 + \im t)$ considered as an element of $W^{\alpha/2}(2I)$. We now argue essentially in the same way as in the proof of Theorem \ref{main result}.
	Since $R_{2I} :  \ell^2_\alpha \rightarrow W^{\alpha/2}(2I)$ is surjective, there exists a 
	sequence $(\gamma_n)_{n \in \Z^\ast}$ such that 
	\begin{equation*}
		v = \sum_{n \in \N} (\gamma_n n^{-1/2-\im t} + \gamma_{-n} n^{-1/2 + \im t}),
	\end{equation*}
	where may assume that $\gamma_{-n}=\overline{\gamma_n}$,
	and  the convergence takes place in  $W^{\alpha/2}(2I)$. It now follows that the function
	\begin{equation*}
		F(s)  =  2\sum_{n \in \N} \gamma_nn^{-s}
	\end{equation*}
	is in $\Hp^2_\alpha$ and satisfies 
	\begin{equation*}
		\lim_{\sigma \rightarrow 1/2^+} \Re F(\sigma + \im t) = v	 
	\end{equation*}
	in the sense of distributions on $2I$. Hence, the function $F-f$ is analytic on $\C_{1/2}$, and has vanishing real parts on $2I$ in the sense of distributions. 
	
	The analytic continuation of the function $F-f$ to all of $\C_I$ is 
	now standard and is obtained in much the same way as in Theorem \ref{main result}. E.g., one may consider a conformal map
	$g:\mathbb{D}\to \Omega,$ where $\Omega\subset\{ \sigma >1/2\}$ is a smooth domain having $3I$ as a boundary segment, whence
	$g$ itself continues analytically over $3I$.
	It follows that the analytic function $(F-f)\circ g$ satisfies
	a bound $|(F-f)\circ g(z)|\leq (1-|z|)^{-\beta}$ for some $\beta >0,$ whence its distributional boundary values are well-defined and one
	obtains it (up to a constant) as the Szeg\"o-integral of the distributional boundary values
	of its real part. The desired analytic continuation follows immediately. 
	
	The assertion about the smallest element and the existence of a norm constant follow exactly as in the proof of Theorem \ref{main result}.
\end{proof}

\section*{acknowledgements}
	This paper is part of the PhD thesis of the first author, who would like to thank his supervisor K.~Seip for helpfull discussions.

\bibliographystyle{amsplain}
\bibliography{bibliotek}

\end{document}